\documentclass[11pt,reqno]{amsart}

\usepackage[T1]{fontenc}
\usepackage{lmodern}
\usepackage{microtype}
\usepackage{amsmath,amssymb,amsfonts,mathtools,mathrsfs}
\usepackage{amsthm}
\usepackage{enumitem}
\usepackage[letterpaper,hmargin=1.25in,vmargin=1in,includeheadfoot]{geometry}
\usepackage{booktabs}
\usepackage{array}
\usepackage{xcolor}
\usepackage{hyperref}
\usepackage[nameinlink,capitalise,noabbrev]{cleveref}

\hypersetup{
  hidelinks,
  pdfauthor={AUTHOR NAME(S)},
  pdftitle={Qin's quasimodularity conjecture for Hilbert schemes of points}
}

\numberwithin{equation}{section}
\allowdisplaybreaks

\newtheorem{theorem}{Theorem}[section]
\newtheorem{proposition}[theorem]{Proposition}
\newtheorem{lemma}[theorem]{Lemma}
\newtheorem{corollary}[theorem]{Corollary}
\newtheorem{conjecture}[theorem]{Conjecture}
\newtheorem{definition}[theorem]{Definition}
\newtheorem{remark}[theorem]{Remark}

\newtheorem*{definition*}{Definition}

\crefname{proposition}{proposition}{propositions}
\Crefname{proposition}{Proposition}{Propositions}
\crefname{lemma}{lemma}{lemmas}
\Crefname{lemma}{Lemma}{Lemmas}
\crefname{corollary}{corollary}{corollaries}
\Crefname{corollary}{Corollary}{Corollaries}
\crefname{conjecture}{conjecture}{conjectures}
\Crefname{conjecture}{Conjecture}{Conjectures}
\crefname{definition}{definition}{definitions}
\Crefname{definition}{Definition}{Definitions}
\crefname{remark}{remark}{remarks}
\Crefname{remark}{Remark}{Remarks}
\crefname{example}{example}{examples}
\Crefname{example}{Example}{Examples}

\crefname{enumi}{step}{steps}
\Crefname{enumi}{Step}{Steps}

\newlist{correlationsteps}{enumerate}{1}
\setlist[correlationsteps,1]{label=\textup{Step F\arabic*.},ref=F\arabic*,leftmargin=55pt}
\crefname{correlationstepsi}{step}{steps}
\Crefname{correlationstepsi}{Step}{Steps}
\newlist{tautologicalsteps}{enumerate}{1}
\setlist[tautologicalsteps,1]{label=\textup{Step T\arabic*.},ref=T\arabic*,leftmargin=55pt}
\crefname{tautologicalstepsi}{step}{steps}
\Crefname{tautologicalstepsi}{Step}{Steps}

\newcommand{\HH}{\mathbb H_X}
\newcommand{\Q}{\mathbb Q}
\newcommand{\C}{\mathbb C}
\newcommand{\Z}{\mathbb Z}
\newcommand{\id}{\operatorname{Id}}
\newcommand{\ch}{\operatorname{ch}}
\newcommand{\td}{\operatorname{td}}
\newcommand{\Tr}{\operatorname{Tr}}
\newcommand{\Str}{\operatorname{Str}}
\newcommand{\CT}{\operatorname{CT}}
\newcommand{\QM}{\mathrm{QM}}
\newcommand{\Fil}{\operatorname{Fil}}

\newcommand{\eps}{\varepsilon}
\newcommand{\la}{\langle}
\newcommand{\ra}{\rangle}
\newcommand{\Nak}{\mathfrak a}
\newcommand{\GG}{\mathfrak G}
\newcommand{\nn}{\mathfrak n}
\newcommand{\WW}{\mathbf W}
\newcommand{\Afield}{\mathbf a}
\newcommand{\Vcur}{\mathcal V}
\newcommand{\Scur}{\mathcal S}
\newcommand{\Ccur}{\mathcal C}
\newcommand{\QEll}{\mathcal Q}
\newcommand{\zhat}{\widehat Z}

\newcommand{\1}{1_X}
\newcommand{\eX}{e_X}
\newcommand{\KX}{K_X}
\newcommand{\qpoch}{(q;q)_\infty}
\newcommand{\Dq}{q\frac{\mathrm d}{\mathrm d q}}
\newcommand{\OZ}[1]{\mathbf Z(#1)}

\title[Qin's quasimodularity conjecture]{Qin's quasimodularity conjecture\\for Hilbert schemes of points}
\author{Victor Alekseev, Avik Chakravarty, Daebeom Choi and Shengjing Xu}
\date{August 21, 2026}

\begin{document}

\begin{abstract}
Let \(X\) be a smooth projective complex surface with numerically trivial canonical class. With the help of GPT-5.6 Sol, we prove Qin's conjecture that the reduced generating series of intersection numbers of Chern characters of tautological bundles against the total Chern class of \(X^{[n]}\) is a quasimodular form with the predicted mixed-weight bound. The main ingredient is a Wick's theorem-type formula for computing traces of normally ordered products of Nakajima operators on \(\bigoplus_n H^\ast(X^{[n]})\). Following the argument of Li--Qin--Wang, the computation of the reduced generating series is reduced to computing the constant term of the supertrace of a product of certain operator-valued currents and the Carlsson--Okounkov operator. Our trace formula shows that this supertrace can be expressed in terms of two quasi--elliptic functions \(\widehat{Z}, P\) and a quasimodular function \(T\), whose constant terms are quasimodular forms by the theorem of Goujard--M\"oller. Finally, we give an explicit algorithm for computing the general reduced generating series.
\end{abstract}

\maketitle

\section{Introduction}\label{sec:introduction}

Let $X$ be a smooth projective complex surface and let $X^{[n]}$ denote the Hilbert scheme of zero-dimensional subschemes of length $n$.  For a line bundle $L$ on $X$, the universal family determines a rank-$n$ tautological bundle $L^{[n]}$ on $X^{[n]}$.  Following Okounkov~\cite{Okounkov2014}, for line bundles $L_1,\ldots,L_N$ and nonnegative integers $k_1,\ldots,k_N$, consider
\begin{equation}\label{eq:intro-series}
\left\langle
  \ch_{k_1}^{L_1}\cdots\ch_{k_N}^{L_N}
\right\rangle
=
\sum_{n\geq 0}q^n
\int_{X^{[n]}}
\prod_{i=1}^N \ch_{k_i}\!\left(L_i^{[n]}\right)
\,c\!\left(T_{X^{[n]}}\right),
\end{equation}
where $c(T_{X^{[n]}})$ is the total Chern class.  Its reduced form is
\begin{equation}\label{eq:intro-reduced}
\left\langle
  \ch_{k_1}^{L_1}\cdots\ch_{k_N}^{L_N}
\right\rangle'
=
\qpoch^{\chi(X)}
\left\langle
  \ch_{k_1}^{L_1}\cdots\ch_{k_N}^{L_N}
\right\rangle .
\end{equation}
Here $(q;q)_\infty=\prod_{m\geq1}(1-q^m)$ and
$\chi(X)=\sum_j(-1)^j\dim H^j(X,\Q)$ is the topological Euler
characteristic.  The normalization is natural because G\"ottsche's formula~\cite{Gottsche1990}, gives
\begin{equation}\label{eq:gottsche-product}
\sum_{n\geq0}q^n\int_{X^{[n]}}c(T_{X^{[n]}})
=\qpoch^{-\chi(X)}.
\end{equation}

Write
\[
\QM=\Q[E_2,E_4,E_6]
\]
for the ring of level-one quasimodular forms, where $E_2,E_4,E_6$
are the Eisenstein series normalized to have constant coefficient $1$,
with weights $2,4,6$.  Since the series arising here need not be homogeneous, we use the mixed-weight filtration
\[
\Fil_{\leq W}\QM
=\bigoplus_{0\leq w\leq W}\QM_w.
\]
Here $\QM_w$ denotes the subspace of homogeneous weight $w$.
When $W$ is odd, this is the same as $\Fil_{\leq W-1}\QM$.

The following quasimodularity conjecture, formulated by Qin for surfaces with numerically trivial canonical class, is closely related to Okounkov's multiple-$q$-zeta-value conjecture; see~\cite{Qin2018,AlhwaimelQin2024,Alhwaimel2025}.

\begin{conjecture}\label[conjecture]{conj:qin}
Assume that $K_X$ is numerically trivial.  Then
\begin{equation}\label{eq:main-theorem}
\left\langle
  \ch_{k_1}^{L_1}\cdots\ch_{k_N}^{L_N}
\right\rangle'
\in
\Fil_{\leq\sum_{i=1}^N(k_i+2)}\QM.
\end{equation}
\end{conjecture}

The main result of this paper establishes the conjecture in full generality.

\begin{theorem}\label{thm:main}
\Cref{conj:qin} holds.
\end{theorem}

The relation between Hilbert-scheme integrals and modular objects has several antecedents.  The Heisenberg action of Grojnowski and Nakajima~\cite{Grojnowski1996,Nakajima1997} organizes the cohomology of all Hilbert schemes into a Fock representation.  Li, Qin, and Wang described the cup-product operators associated with universal Chern characters in this representation~\cite{LiQinWang2002a,LiQinWang2002b}.  Carlsson and Okounkov constructed the Ext vertex operator~\cite{CarlssonOkounkov2012}, while Carlsson obtained vertex-operator quasimodularity results for Chern numbers~\cite{Carlsson2012}.  Qin and Yu proved Okounkov's assertion modulo lower-weight terms and treated the abelian-surface case~\cite{QinYu2018}.  Shen and Qin established further quasimodularity results~\cite{ShenQin2020}.  More recently, Alhwaimel and Qin treated the series $\langle\ch_2^L\rangle'$~\cite{AlhwaimelQin2024}, and Alhwaimel treated the series $\langle\ch_1^{L_1}\ch_1^{L_2}\rangle'$~\cite{Alhwaimel2025}.

The proof begins with the Grothendieck--Riemann--Roch theorem (GRR).  For $\alpha\in H^*(X,\Q)$, let $G_j(\alpha,n)$ denote the component of degree $\deg(\alpha)+2j$ of the universal class $G(\alpha,n)$ defined in \eqref{eq:G-total}.  Then \Cref{prop:GRR-decomposition} gives
\begin{equation}\label{eq:intro-grr}
\ch_k(L^{[n]})
=G_k(\1,n)+G_{k-1}(c_1(L),n)
+G_{k-2}\!\left(\frac{c_1(L)^2}{2},n\right),
\end{equation}
with negative-index terms omitted.  Expanding every factor in \eqref{eq:intro-series} by \eqref{eq:intro-grr} reduces Qin's series to a finite linear combination of generating series whose integrands contain products of the classes $G_j(\alpha,n)$.  In the mixed-weight estimate of \Cref{thm:G-insertions}, a factor $G_j(\alpha,n)$ contributes $j+2$ to the total bound $\sum_i(j_i+2)$.  Thus the three terms on the right-hand side of \eqref{eq:intro-grr} contribute $k+2$, $k+1$, and $k$, respectively, to that total bound.

For $k\geq0$ and $\alpha\in H^*(X,\Q)$, let $\GG_k(\alpha)$ be the operator acting on $H^*(X^{[n]})$ by cup product with $G_k(\alpha,n)$.  For any ordered collections $j_1,\ldots,j_N$ and $\alpha_1,\ldots,\alpha_N$, \Cref{thm:CO-trace} identifies the series
\[
\sum_{n\geq0}q^n\int_{X^{[n]}}
\left(\prod_{i=1}^NG_{j_i}(\alpha_i,n)\right)c(T_{X^{[n]}})
\]
with
\[
\Str_{\HH}\!\left(q^{\nn}\WW_y
\prod_{i=1}^N\GG_{j_i}(\alpha_i)\right).
\]
Dividing by the trace without cup-product operators, computed in \eqref{eq:no-insertion-trace}, gives the normalized trace functional \eqref{eq:Omega-def}. To evaluate this trace, we use the Li--Qin--Wang formula \cite[Theorem~4.6]{LiQinWang2002b}, recalled in \Cref{thm:LQW}, to expand each \(\GG_{j_i}(\alpha_i)\) as a sum of products of Nakajima operators. By \Cref{lem:K-rational-zero}, the surviving terms can then be written, using \Cref{prop:current-reduction}, as constant terms of the operator-valued series \(\Vcur_r\) and \(\Scur_r\) in \eqref{eq:current-reduction}, whose coefficients are normally ordered products of Nakajima operators.

The core of the proof is therefore to compute the supertraces of normally ordered products of Nakajima operators. This will in turn yield the supertraces of products of the operator-valued currents \(\Vcur_r\) and \(\Scur_r\). Recall that the Nakajima operators form a Heisenberg algebra, the operator algebra of a free quantum field theory. We can therefore follow the methodology of Wick's theorem in QFT: first compute the supertraces of products of Nakajima operators in \Cref{lem:shifted-trace,prop:graded-wick-recursion}, and then deduce the supertraces of normally ordered products by induction in \Cref{lem:normal-ordering-subtraction}. This leads to \Cref{thm:wick-current}, which gives the desired formula for supertraces of products of operator-valued currents.

Altogether, as proved in \Cref{subsec:wick-graphs}, the supertrace is expressed as a sum parametrized by Wick graphs,
\[
\sum_\Gamma c_\Gamma\,
\mathcal I_\Gamma\,\mathscr F_\Gamma,
\]
where \(c_\Gamma\in\Q\) comes from the current reduction, \(\mathcal I_\Gamma\) is the \(q\)-independent cohomological contribution defined in \Cref{def:cohomological-contribution}, and \(\mathscr F_\Gamma\) is the associated scalar function introduced in \Cref{def:wick-graph}. The key point is that the only scalar factors appearing in \(\mathscr F_\Gamma\) are derivatives of the normalized zeta function \(\zhat\) and its derivative \(P\), together with the divisor sums \(T_s\), as displayed in \eqref{eq:wick-current-formula}. Hence, \(\mathscr F_\Gamma\) is a quasi-elliptic function of the type considered in \Cref{sec:quasielliptic}. The constant-term theorem of Goujard--M\"oller \cite[Theorem~5.8]{GoujardMoller2020}, recalled in \Cref{thm:GM}, then shows that its chamber constant term is quasimodular with the required weight bound.

We emphasize that the use of Wick's theorem in the computation of reduced generating functions has precedent in the literature. Li–Qin–Wang \cite[Section~6]{LiQinWang2002b} explicitly used Wick's theorem to compute OPEs and establish the \(W\)-algebra commutation relations. More directly related to our calculations are \cite[Section~4]{QinYu2018} and \cite[Lemmas~4.5–4.7]{Alhwaimel2025}, whose trace-reduction methods are closely related to those used in our \Cref{lem:shifted-trace}, \Cref{prop:graded-wick-recursion} and \Cref{lem:normal-ordering-subtraction}. In particular, \cite[Theorem~4.10]{QinYu2018} expresses the generating series as constant terms of expressions involving multiple \(q\)-zeta values with auxiliary variables. Our contribution is to organize the full trace expansion in terms of Wick graphs and to carry the resulting scalar-kernel computations through to quasi-elliptic functions. This places the reduced generating function in a form to which the Goujard–M\"oller constant-term theorem applies, yielding the desired quasimodularity.

All in all, we prove the following more general theorem, from which \Cref{thm:main} essentially follows by the GRR decomposition \eqref{eq:intro-grr}.

\begin{theorem}\label{thm:G-insertions}
Assume that $K_X$ is numerically trivial.  Let $\alpha_1,\ldots,\alpha_N\in H^*(X,\Q)$ and $j_1,\ldots,j_N\geq0$.  Then
\begin{equation}\label{eq:G-insertion-result}
\qpoch^{\chi(X)}
\sum_{n\geq0}q^n
\int_{X^{[n]}}
\left(\prod_{i=1}^N G_{j_i}(\alpha_i,n)\right)
c(T_{X^{[n]}})
\in
\Fil_{\leq\sum_{i=1}^N(j_i+2)}\QM.
\end{equation}
\end{theorem}

The proof of \Cref{thm:G-insertions} is completed in \Cref{sec:quasielliptic}.  \Cref{sec:main-proof} then expands each tautological Chern character by \eqref{eq:intro-grr} and gives the detailed deduction of \Cref{thm:main}.

The paper is organized as follows.  \Cref{sec:operators} develops the current reduction, \Cref{sec:wick} proves the graded Wick theorem, and \Cref{sec:quasielliptic} applies the Goujard--M\"oller theorem.  \Cref{sec:main-proof} passes from universal classes to tautological Chern characters.  \Cref{sec:algorithm} records the resulting finite algorithm, \Cref{sec:checks} applies it to low-degree examples.

\section{Hilbert-scheme operators and current reduction}\label{sec:operators}

The goal of this section is to realize the universal Chern character operators \(\GG_k(\alpha)\) as the constant terms of certain operator-valued currents. This is achieved in \Cref{prop:current-reduction}, which reformulates results of \cite{LiQinWang2002b} in the language of current algebras. Thus, this section may be viewed as a reinterpretation of some results of \cite{LiQinWang2002b} in our framework.

\subsection{The super-Heisenberg Fock space}\label{subsec:super-heisenberg}
Here, we review some basic notions and notation. Let $X$ be a smooth projective complex surface and let $X^{[n]}$ denote the Hilbert scheme of zero-dimensional schemes of length $n$. Consider the super-vector space
\[
H=H^*(X,\Q)=H_{\bar0}\oplus H_{\bar1}
\]
with the graded symmetric Poincar\'{e} pairing
\[
(\alpha,\beta)=\int_X\alpha\cup\beta;\qquad
(\alpha,\beta)=(-1)^{|\alpha||\beta|}(\beta,\alpha).
\]
Integration is extended by zero on $H^j(X,\Q)$ for $j\ne4$.
Here $H_{\bar0}=H^{\mathrm{even}}(X,\Q)$ and
$H_{\bar1}=H^{\mathrm{odd}}(X,\Q)$.
For a homogeneous class $\alpha$, write $\deg(\alpha)$ for its cohomological degree and $|\alpha|=\deg(\alpha)\bmod 2$ for its parity. Set
\[
\HH=\bigoplus_{n\geq0}H^*(X^{[n]},\Q).
\]
For an endomorphism $A$ of $\HH$, we use the categorical supertrace
\[
\Str_{\HH}(A)=\sum_{n,j}(-1)^j
\Tr\!\left(A\big|_{H^j(X^{[n]})}\right),
\]
where only the diagonal block in the number grading contributes.

For $m\in\Z\setminus\{0\}$ and $\alpha\in H$, let $\Nak_m(\alpha)$ be the Nakajima operator, with the sign convention as in~\cite{LiQinWang2002b,Alhwaimel2025}.  For homogeneous $\alpha$, the operator $\Nak_m(\alpha)$ has parity $|\alpha|$.  Negative modes are creation operators and positive modes are annihilation operators.  The commutation relation of ~\cite[Theorem~3.1(i)]{LiQinWang2002b} is recorded in our convention as \eqref{eq:heisenberg}:
\begin{equation}\label{eq:heisenberg}
[\Nak_m(\alpha),\Nak_n(\beta)]_{\mathrm s}
=-m\,\delta_{m,-n}(\alpha,\beta)\id,
\end{equation}
where $[A,B]_{\mathrm s}=AB-(-1)^{|A||B|}BA$, and
\begin{equation}\label{eq:number-commutator}
[\nn,\Nak_m(\alpha)]=-m\Nak_m(\alpha).
\end{equation}
Here $\nn$ is the number-of-points operator, so that $\nn|_{H^*(X^{[n]})}=n\operatorname{Id}$; it is denoted by $\mathfrak{L}_0(1)$ in ~\cite[Theorem~3.1]{LiQinWang2002b}. The vacuum vector is $|0\ra=1\in H^0(X^{[0]})$.

Let $\tau_r:X\to X^r$, $x\mapsto (x,\ldots,x)$, be the diagonal embedding.  If
\[
\tau_{r*}\alpha=\sum_\nu \alpha_{\nu,1}\otimes\cdots\otimes\alpha_{\nu,r}
\]
is a homogeneous K\"unneth decomposition, write
\begin{equation}\label{eq:diagonal-operator}
(\Nak_{m_1}\cdots\Nak_{m_r})(\alpha)
=
\sum_\nu
\Nak_{m_1}(\alpha_{\nu,1})\cdots
\Nak_{m_r}(\alpha_{\nu,r}).
\end{equation}
This is independent of the chosen decomposition.  All tensor permutations are interpreted in the symmetric monoidal category of super-vector spaces.

A generalized partition is a finite multiset of nonzero integers,
\[
\lambda=(\cdots(-2)^{n_{-2}}(-1)^{n_{-1}}1^{n_1}2^{n_2}\cdots).
\]
We use
\[
\ell(\lambda)=\sum_i n_i,
\quad
|\lambda|=\sum_i i n_i,
\quad
s(\lambda)=\sum_i i^2 n_i,
\quad
\lambda!=\prod_i n_i!.
\]
For homogeneous modes $A_i=\Nak_{m_i}(\beta_i)$, define the \emph{super normal ordering} by
\begin{equation}\label{eq:super-normal-ordering}
:A_1\cdots A_r:
=\varepsilon(\sigma;\boldsymbol\beta)
A_{\sigma(1)}\cdots A_{\sigma(r)},
\qquad
\varepsilon(\sigma;\boldsymbol\beta)
=\prod_{\substack{i<j\\
\sigma^{-1}(i)>\sigma^{-1}(j)}}
(-1)^{|\beta_i||\beta_j|}.
\end{equation}
Here $\sigma$ stably moves the negative modes to the left of the positive modes and $\varepsilon$ is the Koszul sign; this is the convention of ~\cite[Section~5]{LiQinWang2002b}.
We write simply \emph{normal ordering} for this operation.  A \emph{normally ordered block} is a product enclosed by a single pair of these colons.

Choose any ordering $(\lambda_1,\ldots,\lambda_r)$ of the multiset $\lambda$, where $r=\ell(\lambda)$.  For a homogeneous K\"unneth decomposition
\(
\tau_{r*}\alpha=\sum_\nu
\alpha_{\nu,1}\otimes\cdots\otimes\alpha_{\nu,r}
\), define
\[
\Nak_\lambda(\alpha)
:=
\sum_\nu
:\Nak_{\lambda_1}(\alpha_{\nu,1})\cdots
 \Nak_{\lambda_r}(\alpha_{\nu,r}): .
\]
For inhomogeneous $\alpha$, extend the definition linearly.

\subsection{Universal Chern-character operators}

Let
\[
\mathcal Z_n=\{(\xi,x)\in X^{[n]}\times X \mid x\in \operatorname{Supp}(\xi)\}\subset X^{[n]}\times X
\]
be the universal subscheme, with projections $p_1$ and $p_2$.  For $\alpha\in H$, define
\begin{equation}\label{eq:G-total}
G(\alpha,n)
=p_{1*}\!\left(
  \ch(\mathcal O_{\mathcal Z_n})\cdot\;
  p_2^*\alpha\;\cdot p_2^*\td(X)
\right).
\end{equation}
If $\alpha$ is homogeneous, let $G_k(\alpha,n)$ be the component of cohomological degree $\deg(\alpha)+2k$.  Let $\GG_k(\alpha)$ act on $H^*(X^{[n]})$ by cup product with $G_k(\alpha,n)$.

The following structure theorem is due to Li, Qin, and Wang~\cite[Theorem~4.6]{LiQinWang2002b}; the form used here is recorded in~\cite[Theorem~3.3]{Alhwaimel2025}.

\begin{theorem}\label{thm:LQW}
For $k\geq0$ and $\alpha\in H$,
\begin{align}
\GG_k(\alpha)
={}&-
\sum_{\substack{\ell(\lambda)=k+2\\|\lambda|=0}}
\frac{\Nak_\lambda(\alpha)}{\lambda!}
+
\sum_{\substack{\ell(\lambda)=k\\|\lambda|=0}}
\frac{s(\lambda)-2}{24\lambda!}
\Nak_\lambda(\eX\alpha) \notag\\
&+
\sum_{\substack{\ell(\lambda)=k+1\\|\lambda|=0}}
\frac{g_{1,\lambda}}{\lambda!}
\Nak_\lambda(\KX\alpha)
+
\sum_{\substack{\ell(\lambda)=k\\|\lambda|=0}}
\frac{g_{2,\lambda}}{\lambda!}
\Nak_\lambda(\KX^2\alpha),
\label{eq:LQW}
\end{align}
where $\eX=c_2(T_X)$ and the rational numbers $g_{1,\lambda}$ and $g_{2,\lambda}$ are independent of $X$ and $\alpha$.  In particular,
\begin{equation}\label{eq:G0-explicit}
\GG_0(\alpha)=-\sum_{m>0}(\Nak_{-m}\Nak_m)(\alpha).
\end{equation}
\end{theorem}

\begin{lemma}\label[lemma]{lem:K-rational-zero}
If $K_X$ is numerically trivial, then
\begin{align}
\GG_k(\alpha)
={}-
\sum_{\substack{\ell(\lambda)=k+2\\|\lambda|=0}}
\frac{\Nak_\lambda(\alpha)}{\lambda!}
+
\sum_{\substack{\ell(\lambda)=k\\|\lambda|=0}}
\frac{s(\lambda)-2}{24\lambda!}
\Nak_\lambda(\eX\alpha)
\label{eq:LQW2}
\end{align}
\end{lemma}

\begin{proof}
This follows from the standard fact that, with \(\mathbb{Q}\)-coefficients, numerical equivalence and homological equivalence coincide for divisors.
\end{proof}

\subsection{Two operator-valued currents}

Define the Heisenberg field
\begin{equation}\label{eq:heisenberg-field}
\Afield(\alpha;x)
=\sum_{m\neq0}\Nak_m(\alpha)x^{-m}.
\end{equation}
For $r\geq1$, set
\begin{equation}\label{eq:V-current}
\Vcur_r(\alpha;x)
=
\frac1{r!}
\sum_{m_1,\ldots,m_r\neq0}
:\Nak_{m_1}\cdots\Nak_{m_r}:
(\tau_{r*}\alpha)
\,x^{-\sum_{h=1}^r m_h},
\end{equation}
and
\begin{equation}\label{eq:S-current}
\Scur_r(\alpha;x)
=
\frac1{r!}
\sum_{m_1,\ldots,m_r\neq0}
\left(\sum_{h=1}^r m_h^2\right)
:\Nak_{m_1}\cdots\Nak_{m_r}:
(\tau_{r*}\alpha)
\,x^{-\sum_{h=1}^r m_h}.
\end{equation}
Equivalently, the two currents have the field expressions
\[
\Vcur_r(\alpha;x)
=\frac1{r!}
:\underbrace{\Afield(x)\cdots\Afield(x)}_{r\text{ factors}}:
(\tau_{r*}\alpha)
\]
and
\begin{equation}\label{eq:S-marked-half-edge}
\Scur_r(\alpha;x)
=
\frac1{r!}\sum_{h=1}^r
:\Afield(x)\cdots D_x^2\Afield(x)\cdots\Afield(x):
(\tau_{r*}\alpha),
\end{equation}
where \(D_x=x\frac{\partial}{\partial x}\) and $D_x^2$ acts on the field in the $h$-th position.
Here the fields are applied multilinearly to the diagonal class: if
$\tau_{r*}\alpha=\sum_\nu\alpha_{\nu,1}\otimes\cdots\otimes\alpha_{\nu,r}$
is a homogeneous K\"unneth decomposition, then
\[
:\Afield(x)\cdots\Afield(x):(\tau_{r*}\alpha)
=\sum_\nu
:\Afield(\alpha_{\nu,1};x)\cdots\Afield(\alpha_{\nu,r};x):.
\]
Expanding each field using \eqref{eq:heisenberg-field} recovers
\eqref{eq:V-current} and \eqref{eq:S-current}, since
$D_x^2x^{-m}=m^2x^{-m}$.

We set
\begin{equation}\label{eq:zero-current-convention}
\Vcur_0=\Scur_0=0.
\end{equation}

These operator-valued currents package the two types of Heisenberg sums appearing in the Li--Qin--Wang formula; compare \cite[Definitions~4.1 and~5.1, Lemma~5.3]{LiQinWang2002b}.

\begin{lemma}\label[lemma]{lem:current-constant-terms}
For $r\geq1$,
\begin{align}
[x^0]\Vcur_r(\alpha;x)
&=
\sum_{\substack{\ell(\lambda)=r\\|\lambda|=0}}
\frac{\Nak_\lambda(\alpha)}{\lambda!},
\label{eq:V-CT}\\
[x^0]\Scur_r(\alpha;x)
&=
\sum_{\substack{\ell(\lambda)=r\\|\lambda|=0}}
\frac{s(\lambda)\Nak_\lambda(\alpha)}{\lambda!}.
\label{eq:S-CT}
\end{align}
\end{lemma}

\begin{proof}
The identities follow from the calculation in the proof of ~\cite[Lemma~5.3]{LiQinWang2002b}. For completeness, we give the proof in our notation.

For the first identity, taking the constant term in
\eqref{eq:V-current} gives
\[
[x^0]\Vcur_r(\alpha;x)
=
\frac{1}{r!}
\sum_{\substack{m_1,\ldots,m_r\neq 0\\
m_1+\cdots+m_r=0}}
:\Nak_{m_1}\cdots\Nak_{m_r}:
(\tau_{r*}\alpha)
=\sum_{\substack{\ell(\lambda)=r\\|\lambda|=0}}
\frac{\Nak_\lambda(\alpha)}{\lambda!},
\]
since the number of times \(\Nak_\lambda(\alpha)\) appears in the middle term is
\[
\binom{\sum_i n_i}{(n_i)_i}=\frac{r!}{\lambda!},
\]
where \(\lambda=(\cdots(-2)^{n_{-2}}(-1)^{n_{-1}}1^{n_1}2^{n_2}\cdots)\).

For the second identity, we similarly have
\[
[x^0]\Scur_r(\alpha;x)
=
\frac{1}{r!}
\sum_{\substack{m_1,\ldots,m_r\neq0\\
m_1+\cdots+m_r=0}}
\left(\sum_{h=1}^r m_h^2\right)
:\Nak_{m_1}\cdots\Nak_{m_r}:
(\tau_{r*}\alpha)
=\sum_{\substack{\ell(\lambda)=r\\|\lambda|=0}}
\frac{s(\lambda)\Nak_\lambda(\alpha)}{\lambda!}.
\]

\end{proof}

The following result is the main result of this section and may be viewed as a current-algebra reformulation of the zero-mode realization of the Chern character operators in \cite[Lemma~5.3, Lemma~5.4]{LiQinWang2002b}.

\begin{proposition}[Current reduction]\label[proposition]{prop:current-reduction}
Assume that $K_X$ is numerically trivial.  Then for every $k\geq0$ and $\alpha\in H$,
\begin{equation}\label{eq:current-reduction}
\GG_k(\alpha)
=-[x^0]\Vcur_{k+2}(\alpha;x)
+\frac1{24}[x^0]
\bigl(
  \Scur_k(\eX\alpha;x)-2\Vcur_k(\eX\alpha;x)
\bigr).
\end{equation}
\end{proposition}

\begin{proof}
The Li--Qin--Wang formula \Cref{thm:LQW}, together with
\Cref{lem:K-rational-zero,lem:current-constant-terms}, gives the following
constant-term reformulation. 
\begin{align*}
\GG_k(\alpha)
&=-
\sum_{\substack{\ell(\lambda)=k+2\\|\lambda|=0}}
\frac{\Nak_\lambda(\alpha)}{\lambda!}
+\frac1{24}
\sum_{\substack{\ell(\lambda)=k\\|\lambda|=0}}
\frac{\bigl(s(\lambda)-2\bigr)\Nak_\lambda(\eX\alpha)}{\lambda!}
\\
&=-[x^0]\Vcur_{k+2}(\alpha;x)
+\frac1{24}\left(
[x^0]\Scur_k(\eX\alpha;x)
-2[x^0]\Vcur_k(\eX\alpha;x)
\right),
\end{align*}
which is \eqref{eq:current-reduction}.
\end{proof}
\section{The graded Wick expansion}\label{sec:wick}

This section is the heart of the proof. As mentioned in \Cref{sec:introduction} and reviewed in \Cref{sec:main-proof}, the generating function appearing in Qin's conjecture is a linear combination of generating functions of intersection numbers of the universal classes \(G_j(\alpha,n)\) against the total Chern class. In \Cref{prop:reduced-trace}, we relate these generating functions to the supertraces of the operators \(\GG_j(\alpha)\) against the Carlsson--Okounkov operator \(\WW_y\). \Cref{prop:current-reduction} further relates these supertraces to the constant terms of supertraces of operator-valued currents. The key point is that the coefficients of these operator-valued currents are normally ordered products of Nakajima operators.

Therefore, Qin's conjecture is essentially reduced to the computation of supertraces of normally ordered products of Nakajima operators. Since the Nakajima operators form a Heisenberg algebra, the operator algebra of a free QFT, we are able to prove an analogue of Wick's theorem, \Cref{lem:normal-ordering-subtraction}. As in Wick's theorem, we first compute the supertrace of an ordinary product of Nakajima operators in \Cref{lem:shifted-trace,prop:graded-wick-recursion}, and then deduce the supertrace of a normally ordered product by induction. The corresponding formula for normally ordered products of currents then follows in \Cref{thm:wick-current}. The upshot is that the final expression in \Cref{thm:wick-current} involves only three functions, which are reviewed in \Cref{subsec:scalar-kernels}.

Note that, as mentioned in the introduction, the idea of using Wick's theorem to compute the generating function already existed in the literature; see \cite{LiQinWang2002b, QinYu2018, Alhwaimel2025}. In particular, the computations in \Cref{subsec:ext-trace}--\Cref{subsec:normal-ordering} are based on the corresponding computations in \cite[Section~4]{QinYu2018} and \cite[Lemmas~4.5--4.7]{Alhwaimel2025}.

\subsection{The Carlsson--Okounkov trace}\label{subsec:ext-trace}

Let $T=\C^*$ act trivially on $X$, let $\C_m$ denote the one-dimensional $T$-representation of weight $m$, and put
\[
t=c_1^T(\C_1)\in H_T^2(\operatorname{pt},\Q).
\]
The equivariant line bundle used in the tangent-class trace is
\begin{equation}\label{eq:L1-equivariant}
\mathfrak L_1=\mathcal O_X\otimes\C_1,
\qquad
c_1^T(\mathfrak L_1)=t\1.
\end{equation}
Thus $\mathfrak L_1$ is trivial after forgetting the $T$-linearization, but it is not equivariantly trivial.

For a $T$-equivariant line bundle $\mathfrak L$, write
$\ell=c_1^T(\mathfrak L)$.  In the Heisenberg convention of ~\cite[Equation~(3.6)]{QinYu2018}, the Carlsson--Okounkov Ext vertex operator~\cite[Theorem~1]{CarlssonOkounkov2012} is
\begin{equation}\label{eq:W-general}
\WW_T(\mathfrak L,y)
=\Gamma_-(\ell-K_X,y)\Gamma_+(-\ell,y),
\end{equation}
where, in our Heisenberg convention,
\begin{equation}\label{eq:Gamma-def}
\Gamma_\pm(\gamma,y)
=
\exp\!\left(
  \sum_{n>0}\frac{y^{\mp n}}n\Nak_{\pm n}(\gamma)
\right).
\end{equation}
For \eqref{eq:L1-equivariant}, this gives
\begin{equation}\label{eq:W-L1-equivariant}
\WW_T(\mathfrak L_1,y)
=\Gamma_-(t\1-K_X,y)\Gamma_+(-t\1,y).
\end{equation}
Following the convention in ~\cite[Section~3]{QinYu2018}, we specialize the equivariant generator to $t=1$:
\begin{equation}\label{eq:W-specialized}
\WW_y
:=\left.\WW_T(\mathfrak L_1,y)\right|_{t=1}
=\Gamma_-(\1-K_X,y)\Gamma_+(-\1,y).
\end{equation}

The following is the supertrace form of ~\cite[Lemma~3.2]{QinYu2018}, based on the Ext operator of ~\cite[Theorem~1]{CarlssonOkounkov2012}.

\begin{theorem}\label{thm:CO-trace}
For homogeneous classes $\alpha_1,\ldots,\alpha_N$ and nonnegative integers $j_1,\ldots,j_N$, put
\begin{equation}\label{eq:F-G}
F^{\alpha_1,\ldots,\alpha_N}_{j_1,\ldots,j_N}(q)
=
\sum_{n\geq0}q^n
\int_{X^{[n]}}
\left(\prod_{i=1}^NG_{j_i}(\alpha_i,n)\right)
 c(T_{X^{[n]}}).
\end{equation}
Then
\begin{equation}\label{eq:CO-trace}
F^{\alpha_1,\ldots,\alpha_N}_{j_1,\ldots,j_N}(q)
=
\left.\Str_{\HH}\!\left(
q^{\nn}\WW_T(\mathfrak L_1,y)
\prod_{i=1}^N\GG_{j_i}(\alpha_i)
\right)\right|_{t=1}
=\Str_{\HH}\!\left(
q^{\nn}\WW_y
\prod_{i=1}^N\GG_{j_i}(\alpha_i)
\right).
\end{equation}
\end{theorem}

To see explicitly why the specialization in \eqref{eq:CO-trace} produces the total tangent Chern class, let $\mathcal E_{\mathfrak L_1}$ be the Carlsson--Okounkov Ext bundle.  On the diagonal $\Delta_n\subset X^{[n]}\times X^{[n]}$ one has
\[
\left.\mathcal E_{\mathfrak L_1}\right|_{\Delta_n}
=T_{X^{[n]}}\otimes\C_1.
\]
Consequently,
\begin{equation}\label{eq:equivariant-top-chern}
c_{2n}^T\!\left(T_{X^{[n]}}\otimes\C_1\right)
=\sum_{i=0}^{2n}c_i(T_{X^{[n]}})t^{2n-i},
\qquad
\left.c_{2n}^T\!\left(T_{X^{[n]}}\otimes\C_1\right)\right|_{t=1}
=c(T_{X^{[n]}}).
\end{equation}
Together with the diagonal decomposition of ~\cite[Lemma~3.2]{QinYu2018}, this proves the trace identity \eqref{eq:CO-trace}.  If one replaced $\mathfrak L_1$ by the weight-zero equivariant bundle $\mathcal O_X\otimes\C_0$, then the left-hand side of \eqref{eq:equivariant-top-chern} would instead be $c_{2n}(T_{X^{[n]}})$; the trace would no longer compute the total-Chern-class series \eqref{eq:F-G}.

The symbol $\Str$ in \eqref{eq:CO-trace} is the usual supertrace on super-vector spaces.  The trace in ~\cite[Lemma~3.2]{QinYu2018} is written $\Tr$; the diagonal decomposition there carries the sign $(-1)^{|e_i|}$ (for a basis $\{e_i\}$ of $H^*(X^{[n]})$, as in the cited source) and therefore gives the supertrace in the graded setting.  This convention is also forced by the identity with no cup-product operator in \eqref{eq:no-insertion-trace}: odd cohomology contributes to the topological Euler characteristic with a minus sign.

From now on, we assume that \(K_X\) is numerically trivial, and hence also trivial in cohomology with \(\mathbb{Q}\)-coefficients. Then \eqref{eq:W-specialized} becomes
\begin{equation}\label{eq:W-Kzero}
\WW_y
=
\Gamma_-(\1,y)\Gamma_+(-\1,y)
=e^{A_-}e^{A_+},
\end{equation}
where
\begin{equation}\label{eq:Apm}
A_-=
\sum_{n>0}\frac{y^n}{n}\Nak_{-n}(\1),
\qquad
A_+=-
\sum_{n>0}\frac{y^{-n}}{n}\Nak_n(\1).
\end{equation}
Since $(\1,\1)=\int_X1=0$ for degree reasons, the two vertex-operator exponentials have no mutual normal-ordering scalar.

Taking $N=0$ in \eqref{eq:CO-trace} and applying G\"ottsche's formula \eqref{eq:gottsche-product} gives the product for the trace with the vertex operator:
\begin{equation}\label{eq:no-insertion-trace}
\Str_{\HH}(q^{\nn}\WW_y)=\qpoch^{-\chi(X)}.
\end{equation}
Define the normalized functional
\begin{equation}\label{eq:Omega-def}
\Omega_{q,y}(A)
=
\frac{\Str_{\HH}(q^{\nn}\WW_yA)}
{\Str_{\HH}(q^{\nn}\WW_y)}.
\end{equation}

The following proposition is a direct consequence of \Cref{thm:CO-trace} and \eqref{eq:Omega-def}, which we record for later use.

\begin{proposition}\label[proposition]{prop:reduced-trace}
For homogeneous $\alpha_i$ and $j_i\geq0$, let $F^{\boldsymbol\alpha}_{\mathbf j}$ be the series in \eqref{eq:F-G} and define its reduced form by
\[
\widehat F^{\boldsymbol\alpha}_{\mathbf j}(q)
:=\qpoch^{\chi(X)}F^{\boldsymbol\alpha}_{\mathbf j}(q).
\]
Then
\begin{equation}\label{eq:reduced-trace}
\widehat F^{\boldsymbol\alpha}_{\mathbf j}(q)
=\Omega_{q,y}\!\left(\prod_{i=1}^N\GG_{j_i}(\alpha_i)\right).
\end{equation}
In particular, this normalized trace is independent of $y$.
\end{proposition}

\subsection{The shifted trace relation and centering}

Put $\rho=q^{\nn}\WW_y$.  By the parity convention of \Cref{subsec:super-heisenberg}, the modes of $\1\in H^0(X)$ and the operators $A_\pm$ in \eqref{eq:Apm} are even.  Their exponentials and $\WW_y$ are therefore even.  Since $q^{\nn}$ also preserves parity, $\rho$ is even.  We first compute its commutation with one Heisenberg mode.

\begin{lemma}\label[lemma]{lem:shifted-trace}
For $m\neq0$ and $\alpha\in H$,
\begin{equation}\label{eq:shifted-trace}
\rho\Nak_m(\alpha)
=
\left(q^{-m}\Nak_m(\alpha)+y^m\int_X\alpha\right)\rho.
\end{equation}
Consequently,
\begin{equation}\label{eq:mode-mean}
\Omega_{q,y}(\Nak_m(\alpha))
=\mu_m(\alpha)
:=\frac{y^m}{1-q^{-m}}\int_X\alpha.
\end{equation}
\end{lemma}

\begin{proof}
By \eqref{eq:heisenberg} and \eqref{eq:Apm},
\[
[A_-,\Nak_n(\alpha)]=y^n\int_X\alpha,
\qquad
[A_+,\Nak_{-n}(\alpha)]=y^{-n}\int_X\alpha
\quad(n>0).
\]
These are scalars, so the Baker--Campbell--Hausdorff expansion terminates and
\[
\WW_y\Nak_m(\alpha)\WW_y^{-1}
=\Nak_m(\alpha)+y^m\int_X\alpha.
\]
Equation \eqref{eq:number-commutator} gives
$q^{\nn}\Nak_mq^{-\nn}=q^{-m}\Nak_m$, proving \eqref{eq:shifted-trace}.  Apply $\Str_{\HH}$ to both sides of \eqref{eq:shifted-trace}, and divide by $\Str_{\HH}(\rho)$. Since \(\rho\) is even, \(\Str(\rho\Nak_m(\alpha))=\Str(\Nak_m(\alpha)\rho)\), and hence, by \eqref{eq:Omega-def}, this gives
\[
\Omega_{q,y}(\Nak_m(\alpha))
=q^{-m}\Omega_{q,y}(\Nak_m(\alpha))
+y^m\int_X\alpha.
\]
Solving gives \eqref{eq:mode-mean}.  For odd $\alpha$, both sides are zero: $\int_X\alpha=0$, and an odd operator has zero supertrace.
\end{proof}

\subsection{The graded Wick recursion}

Define \emph{centered} modes
\begin{equation}\label{eq:centered-mode}
\widetilde\Nak_m(\alpha)
=\Nak_m(\alpha)-\mu_m(\alpha)\id.
\end{equation}
Then \eqref{eq:shifted-trace} becomes homogeneous:
\begin{equation}\label{eq:centered-trace}
\rho\widetilde\Nak_m(\alpha)
=q^{-m}\widetilde\Nak_m(\alpha)\rho.
\end{equation}

Centering preserves the parity of a homogeneous mode. The next proposition shows that a product of an odd number of centered modes has zero normalized trace, whereas an even product is evaluated by pairings. In particular, this computes \(\Omega_{q,y}\) for arbitrary products of Nakajima operators.

\begin{proposition}[Graded Wick recursion]\label[proposition]{prop:graded-wick-recursion}
Let
$b_i=\widetilde\Nak_{m_i}(\alpha_i)$ for $1\leq i\leq r$,
with the $\alpha_i$ homogeneous.  Then
\begin{equation}\label{eq:graded-wick-recursion}
\Omega_{q,y}(b_1\cdots b_r)
=
\sum_{j=2}^r
(-1)^{|\alpha_1|(|\alpha_2|+\cdots+|\alpha_{j-1}|)}
\frac{[b_1,b_j]_{\mathrm s}}{1-q^{m_1}}
\cdot
\Omega_{q,y}(b_2\cdots\widehat b_j\cdots b_r).
\end{equation}
\begin{samepage}
In particular, every odd centered moment vanishes, and every even centered moment is the sum over complete pairings with the Koszul sign of \Cref{subsec:super-heisenberg}. For opposite indices, they are, with $n>0$,
\begin{align}
\Omega_{q,y}\bigl(
\widetilde\Nak_n(\alpha)
\widetilde\Nak_{-n}(\beta)
\bigr)
&=-\frac{n}{1-q^n}(\alpha,\beta),
\label{eq:covariance-plusminus}\\*
\Omega_{q,y}\bigl(
\widetilde\Nak_{-n}(\alpha)
\widetilde\Nak_n(\beta)
\bigr)
&=-\frac{nq^n}{1-q^n}(\alpha,\beta).
\label{eq:covariance-minusplus}
\end{align}
\end{samepage}
\end{proposition}

\begin{proof}
Put $p_i=|\alpha_i|$, $B=b_2\cdots b_r$, and $s=(-1)^{p_1(p_2+\cdots+p_r)}$.
Define
\[
R=
\sum_{j=2}^r
(-1)^{p_1(p_2+\cdots+p_{j-1})}
b_2\cdots b_{j-1}[b_1,b_j]_{\mathrm s}
b_{j+1}\cdots b_r.
\]
Successively supercommuting $b_1$ past $b_2,\ldots,b_r$ gives the identity
\begin{equation*}
b_1B=sBb_1+R,
\qquad
s\,\Omega_{q,y}(Bb_1)
=\Omega_{q,y}(b_1B)-\Omega_{q,y}(R).
\end{equation*}
On the other hand, if $Z=\Str_{\HH}(\rho)$, then \eqref{eq:centered-trace} and graded cyclicity give
\begin{align*}
\Omega_{q,y}(b_1B)
&=Z^{-1}\Str_{\HH}(\rho b_1B)
=q^{-m_1}Z^{-1}\Str_{\HH}(b_1\rho B)\\*
&=q^{-m_1}s\,Z^{-1}\Str_{\HH}(\rho Bb_1)
=q^{-m_1}s\,\Omega_{q,y}(Bb_1).
\end{align*}
Combining the last two displays and writing $A=\Omega_{q,y}(b_1B)$, we obtain
\[
A=q^{-m_1}\bigl(A-\Omega_{q,y}(R)\bigr),
\qquad
A=\frac{-q^{-m_1}}{1-q^{-m_1}}\Omega_{q,y}(R)
=\frac{\Omega_{q,y}(R)}{1-q^{m_1}},
\]
which is \eqref{eq:graded-wick-recursion}.  Since centering changes a mode only by a scalar,
\begin{equation}\label{eq:comm}
[b_1,b_j]_{\mathrm s}
=[\Nak_{m_1}(\alpha_1),\Nak_{m_j}(\alpha_j)]_{\mathrm s}
=-m_1\delta_{m_j,-m_1}(\alpha_1,\alpha_j).
\end{equation}
For $r=2$ and $(m_1,m_2)=(n,-n)$ this gives
\[
\Omega_{q,y}(b_1b_2)
=-\frac{n}{1-q^n}(\alpha_1,\alpha_2),
\]
whereas $(m_1,m_2)=(-n,n)$ gives
\[
\Omega_{q,y}(b_1b_2)
=\frac{n}{1-q^{-n}}(\alpha_1,\alpha_2)
=-\frac{nq^n}{1-q^n}(\alpha_1,\alpha_2).
\]
These are \eqref{eq:covariance-plusminus} and \eqref{eq:covariance-minusplus}.  Finally, \eqref{eq:graded-wick-recursion} removes two centered modes at each step; together with $\Omega_{q,y}(b_i)=0$, induction gives the asserted vanishing and complete-pairing expansion.
\end{proof}

For a partition $\pi$ of a finite ordered set $I$ into singletons and pairs, order each pair increasingly and order all parts by their least elements.  If the original list is $(a_1,\ldots,a_r)$, write the resulting list as $(a_{\sigma_\pi(1)},\ldots,a_{\sigma_\pi(r)})$.  Applying the super-permutation sign of \eqref{eq:super-normal-ordering} gives
\begin{equation}\label{eq:koszul-sign}
\eps(\pi)
=\varepsilon\!\left(\sigma_\pi;(\alpha_{a_1},\ldots,\alpha_{a_r})\right).
\end{equation}
For a partial pairing, include each unpaired label as a singleton.  The complete-pairing expansion in \Cref{prop:graded-wick-recursion} is therefore the sum of $\eps(\pi)$ times the ordered covariances of its pairs.

\subsection{Normal ordering and pairs within one block}\label{subsec:normal-ordering}

The currents in \Cref{prop:current-reduction} are defined using the super normal ordering of \Cref{subsec:super-heisenberg}. Therefore, in this subsection, we compute \(\Omega_{q,y}\) for super normally ordered products of Nakajima operators.

For homogeneous modes $c_a=\Nak_{m_a}(\alpha_a)$ with $m_a\ne0$, put
$\widetilde c_a=\widetilde\Nak_{m_a}(\alpha_a)$ as in
\eqref{eq:centered-mode}, and write
\begin{equation}\label{eq:centered-covariance}
C(c_a,c_b)=\Omega_{q,y}(\widetilde c_a\widetilde c_b).
\end{equation}
By \Cref{prop:graded-wick-recursion} and \eqref{eq:comm}, this ordered covariance vanishes unless \(m_b=-m_a\); its two nonzero
mode orientations are given by
\eqref{eq:covariance-plusminus}--\eqref{eq:covariance-minusplus}.
Its value at $q=0$, in the indicated operator order, is the
\emph{vacuum contraction}
\begin{equation}\label{eq:vacuum-covariance}
C^0(c_a,c_b)
=
\begin{cases}
-m_a(\alpha_a,\alpha_b),
 &m_a>0,\ m_b=-m_a,\\
0,&\text{otherwise}.
\end{cases}
\end{equation}
Equivalently, $C^0(c_a,c_b)=\la0|c_ac_b|0\ra$: a positive mode
annihilates the vacuum, so this matrix element vanishes unless $m_a>0$
and $m_b=-m_a$, in which case it is the supercommutator
\eqref{eq:heisenberg}.  Thus $C^0$ is the contraction of the pair in
the usual sense, and \eqref{eq:normal-ordering-identity} below is the
statement that normal ordering subtracts the vacuum expectation.

\begin{lemma}[Normal-ordering subtraction]\label[lemma]{lem:normal-ordering-subtraction}
Let $c_1,\ldots,c_r$ be homogeneous Heisenberg modes.
\begin{enumerate}[label=(\roman*),leftmargin=25pt]
\item For a single normally ordered block,
\begin{equation}\label{eq:normal-ordering-identity}
:c_1\cdots c_r:
=
\sum_{\rho}
(-1)^{\#\rho}\eps(\rho)
\left(\prod_{\{a,b\}\in\rho}C^0(c_a,c_b)\right)
\prod_{h\notin V(\rho)}^{\longrightarrow}c_h.
\end{equation}
Here $\rho$ ranges over collections of disjoint pairs $\{a,b\}$ with
$a<b$, $\#\rho$ is the number of pairs, and $V(\rho)$ is their union.
The sign $\eps(\rho)$ is \eqref{eq:koszul-sign} for the partition
obtained by adjoining the unpaired indices as singletons.
\item Suppose $H_1,\ldots,H_N$ are consecutive intervals partitioning
$\{1,\ldots,r\}$, and put
$\mathcal O_i=\prod_{a\in H_i}^{\longrightarrow}c_a$.
Normal order each $\mathcal O_i$ separately, retaining the order of the
blocks.  Then
\begin{equation}\label{eq:normal-ordered-block-trace}
\begin{aligned}
\Omega_{q,y}(:\mathcal O_1:\cdots:\mathcal O_N:)
&=\sum_{\Pi}\eps(\Pi)
\prod_{\{a\}\in\Pi}\mu_{m_a}(\alpha_a)\\
&\qquad\times
\prod_{\{a,b\}\in\Pi_{\mathrm{ext}}}C(c_a,c_b)
\prod_{\{a,b\}\in\Pi_{\mathrm{int}}}(C-C^0)(c_a,c_b).
\end{aligned}
\end{equation}
The sum is over partitions $\Pi$ into singletons and pairs, each pair
written with $a<b$.  The subsets $\Pi_{\mathrm{ext}}$ and
$\Pi_{\mathrm{int}}$ consist of pairs in different $H_i$ and in the same
$H_i$, respectively; the means are those of \eqref{eq:mode-mean}.
\end{enumerate}
\end{lemma}

\begin{proof}
Write $c_a=\Nak_{m_a}(\alpha_a)$ and $p_a=|\alpha_a|$.  The Koszul sign
is multiplicative under composition of reorderings, an interchange of
adjacent letters of parities $p,p'$ costing $(-1)^{pp'}$.  Two facts are
used.  \emph{(a)} If $m_a>0$ then $[c_a,c_b]_{\mathrm s}=C^0(c_a,c_b)$
for every $b$, by \eqref{eq:heisenberg} and \eqref{eq:vacuum-covariance};
if $m_a<0$ then $C^0(c_a,c_b)=0$.  \emph{(b)} Both $C(c_a,c_b)$ and
$C^0(c_a,c_b)$ are multiples of $(\alpha_a,\alpha_b)$, so a pair carrying
a nonzero contraction has $\deg\alpha_a+\deg\alpha_b=4$, hence
$p_a=p_b$; it is an even block and crosses any letter freely.  Terms
with a vanishing contraction may be discarded, so \emph{(b)} applies
throughout.  For a list $L$ and $a<b$ in it, let
$\eps_L(a,b)=(-1)^{p_a\sum_{a<h<b}p_h}$, the sign \eqref{eq:koszul-sign}
of the partition of $L$ whose only pair is $\{a,b\}$.

\emph{Step 1.}  For any list $L=(c_1,\ldots,c_r)$,
\begin{equation}\label{eq:one-mode-into-block}
c_1:\!c_2\cdots c_r\!:
\;=\;
:\!c_1c_2\cdots c_r\!:
+\sum_{j=2}^{r}\eps_L(1,j)\,C^0(c_1,c_j)\,
:\!c_2\cdots\widehat{c_j}\cdots c_r\!: .
\end{equation}
Both sides are supersymmetric in $c_2,\ldots,c_r$, so it suffices to
prove them for one order of the tail.  Indeed, interchanging $c_j$ and
$c_{j+1}$ multiplies every normally ordered factor containing both by
$(-1)^{p_jp_{j+1}}$; it exchanges the two summands carrying
$C^0(c_1,c_j)$ and $C^0(c_1,c_{j+1})$, whose normally ordered factors
are unaffected, while their signs change by $(-1)^{p_1p_{j+1}}$ and
$(-1)^{p_1p_j}$, which are again $(-1)^{p_jp_{j+1}}$ because $p_1$
equals the parity of the contracted letter.

Assume then that $c_2,\ldots,c_r$ is normally ordered, with
$c_2,\ldots,c_{k+1}$ the negative modes.  If $m_1<0$, the list
$c_1,\ldots,c_r$ is normally ordered as well and every $C^0(c_1,c_j)$
vanishes by \emph{(a)}, so both sides equal $c_1\cdots c_r$.  If
$m_1>0$, then $c_1c_j=(-1)^{p_1p_j}c_jc_1+C^0(c_1,c_j)$ for
$j\leq k+1$ by \emph{(a)}, and moving $c_1$ across $c_2\cdots c_{k+1}$
one letter at a time gives
\[
c_1c_2\cdots c_r
=(-1)^{p_1(p_2+\cdots+p_{k+1})}c_2\cdots c_{k+1}c_1c_{k+2}\cdots c_r
+\sum_{j=2}^{k+1}\eps_L(1,j)\,C^0(c_1,c_j)\,c_2\cdots\widehat{c_j}\cdots c_r .
\]
The first term on the right is $:\!c_1\cdots c_r\!:$, each product in the
sum is normally ordered, and $C^0(c_1,c_j)=0$ for $j>k+1$ by
\eqref{eq:vacuum-covariance}.

\emph{Step 2.}  Sorting into the order of \eqref{eq:koszul-sign} moves
whole parts, and a pair is an even block.  Let $\rho$ be a set of
disjoint pairs in $\{1,\ldots,r\}$, let $B$ be the increasing list of the
remaining indices, and let $\pi'$ be a partition of $B$.  Sorting into
$\rho$-order and then sorting the sublist $B$, which has kept its
relative order, into $\pi'$-order yields an arrangement of the parts of
$\rho\sqcup\pi'$ whose singletons increase, as they do in the
$(\rho\sqcup\pi')$-order; the two differ by a permutation of even
blocks, so
\begin{equation}\label{eq:eps-multiplicative}
\eps(\rho\sqcup\pi')=\eps(\rho)\,\eps_B(\pi').
\end{equation}

\emph{Step 3.}  We prove \eqref{eq:normal-ordering-identity} by induction
on $r$, the cases $r\leq1$ being trivial.  Solve
\eqref{eq:one-mode-into-block} for $:\!c_1\cdots c_r\!:$ and expand each
normally ordered factor on the right by the inductive hypothesis.  The
term $c_1:\!c_2\cdots c_r\!:$ supplies the terms of
\eqref{eq:normal-ordering-identity} in which $1$ is unpaired: writing
$\rho'$ for a set of disjoint pairs in $\{2,\ldots,r\}$, the part
$\{1\}$ comes first, so $c_1$ crosses nothing and $\eps(\rho')$ is the
required sign.  The $j$-th correction term supplies those with
$\{1,j\}\in\rho$: for a set $\rho''$ of disjoint pairs in
$B_j=\{2,\ldots,r\}\setminus\{j\}$ the contraction $C^0(c_1,c_j)$ raises
the number of pairs by one, matching the sign $-(-1)^{\#\rho''}$, and
$\eps_L(1,j)\,\eps_{B_j}(\rho'')=\eps\bigl(\{\{1,j\}\}\sqcup\rho''\bigr)$
by \eqref{eq:eps-multiplicative}.  Every $\rho$ arises exactly once.  For
$r=2$ the identity reads $:\!c_1c_2\!:\,=c_1c_2-C^0(c_1,c_2)$.

\emph{Step 4.}  For (ii), apply (i) within each $H_i$ and let
$R\subseteq\Pi_{\mathrm{int}}$ be the pairs so removed, of weight
$-C^0$; as the $H_i$ are consecutive intervals, the block signs multiply
to $\eps(R)$.  Center the surviving modes by \eqref{eq:centered-mode} and
apply \Cref{prop:graded-wick-recursion}, which pairs them with weight $C$
and leaves singletons contributing the means \eqref{eq:mode-mean},
nonzero only for even $\alpha_a$ and so extracted without a sign; a pair
in two different $H_i$ can arise only here.  This produces the remaining
parts with the sign $\eps_B(\Pi\setminus R)$, $B$ being the complement of
$V(R)$, so the total sign is $\eps(\Pi)$ by
\eqref{eq:eps-multiplicative}, independent of $R$.  Writing
$C_e=C(c_a,c_b)$ and $C^0_e=C^0(c_a,c_b)$ for $e=\{a,b\}$, the terms
with a given $\Pi$ therefore sum to
\[
\begin{aligned}
&\eps(\Pi)\prod_{\{a\}\in\Pi}\mu_{m_a}(\alpha_a)
\prod_{e\in\Pi_{\mathrm{ext}}}C_e\\
&\qquad\times\sum_{R\subseteq\Pi_{\mathrm{int}}}(-1)^{\#R}
\prod_{e\in R}C^0_e\prod_{e\in\Pi_{\mathrm{int}}\setminus R}C_e ,
\end{aligned}
\]
and the last sum equals $\prod_{e\in\Pi_{\mathrm{int}}}(C_e-C^0_e)$,
proving (ii).
\end{proof}

For $n>0$, \eqref{eq:covariance-plusminus}--\eqref{eq:vacuum-covariance}
give
\[
(C-C^0)(\Nak_{\pm n}(\alpha),\Nak_{\mp n}(\beta))
=-\frac{nq^n}{1-q^n}(\alpha,\beta).
\]
Thus the two mode orientations have the same covariance.  With the even
derivative orders used below, their sum supplies the factor $2$ in the
loop kernel \eqref{eq:loop-kernel}.

\subsection{The scalar kernels and the chamber}\label{subsec:scalar-kernels}

Here, we introduce some (quasi-elliptic) functions that will appear in the main theorem of this section, \Cref{thm:wick-current}. Write
\[
q=e^{2\pi i\tau},\qquad x_i=e^{2\pi iz_i},\qquad y=x_0=e^{2\pi iz_0},
\qquad \Im\tau>0.
\]
We work in the nested chamber
\begin{equation}\label{eq:nested-chamber}
0\leq\Im z_0<\Im z_1<\cdots<\Im z_N<\Im\tau.
\end{equation}
In multiplicative coordinates, these inequalities give
\begin{equation}\label{eq:nested-annuli}
|q|<\left|\frac{x_j}{x_i}\right|<1
\qquad(0\leq i<j\leq N).
\end{equation}
The parameter $y$ of the vertex operator is thus the additional variable
$x_0$.  Let
\[
\Theta(z;\tau)=\sum_{k\in\Z}(-1)^k
q^{(k+1/2)^2/2}e^{2\pi i(k+1/2)z},
\qquad D_z=\frac1{2\pi i}\frac{\partial}{\partial z}.
\]
We use the normalization of Goujard--M\"oller
~\cite[Section~5.2, equations~(40)--(41)]{GoujardMoller2020}:
\begin{equation}\label{eq:ZP-identification}
Z(z)=-\frac{D_z\Theta(z;\tau)}{\Theta(z;\tau)},
\qquad \zhat(z)=Z(z)-\frac12,
\qquad P(z)=D_zZ(z).
\end{equation}
We suppress $\tau$ and, when $u=e^{2\pi iz}$, also write
$\zhat(u)$ and $P(u)$ for these periodic functions.  Their Fourier
expansions in $|q|<|u|<1$ are
\begin{align}
\zhat(u)
&=\sum_{n\geq1}
\left(\frac{u^n}{1-q^n}-\frac{q^nu^{-n}}{1-q^n}\right),
\label{eq:Zhat-Fourier}\\
P(u)
&=\sum_{n\geq1}n
\left(\frac{u^n}{1-q^n}+\frac{q^nu^{-n}}{1-q^n}\right).
\label{eq:P-Fourier}
\end{align}
Here $D_z$ agrees with $D_u=u\partial_u$ in multiplicative coordinates.

For $s\geq2$, put
\begin{equation}\label{eq:loop-series-def}
T_s(q)=\sum_{n\geq1}n^{s-1}\frac{q^n}{1-q^n}.
\end{equation}
When $s$ is even,
\begin{equation}\label{eq:loop-series-Eisenstein}
T_s(q)=\frac{B_s}{2s}\bigl(1-E_s(q)\bigr).
\end{equation}
Here $B_s$ is the Bernoulli number defined by
$t/(e^t-1)=\sum_{s\geq0}B_st^s/s!$, and $E_s$ has constant
coefficient $1$.

For an even nonnegative integer $d$, set
\begin{equation}\label{eq:derived-field}
\Afield^{(d)}(\alpha;x)=D_x^d\Afield(\alpha;x)
\end{equation}
(cf. \eqref{eq:heisenberg-field}). Only $d=0$ and $d=2$ will occur in the application.

We interpret finite-mode traces coefficientwise on the finite-dimensional
summands $H^*(X^{[n]})$ of $\HH$.  For current products, the following lemmas
justify summing those traces over the modes in the chamber
\eqref{eq:nested-annuli}.

\begin{lemma}[Normal convergence]\label[lemma]{lem:normal-convergence}
Every fixed power of $D_u$ applied to \eqref{eq:Zhat-Fourier} or
\eqref{eq:P-Fourier}, and every series $T_s$ in
\eqref{eq:loop-series-def}, converges absolutely and uniformly on compact
subsets of \eqref{eq:nested-annuli}.
\end{lemma}

\begin{proof}
On such a compact set there are constants $r,a,b<1$ such that
$|q|\leq r$, $|u|\leq a$, and $|q/u|\leq b$ for every ratio
$u=x_j/x_i$ with $i<j$.  For each integer $h\geq0$,
\[
\frac{n^h(|u|^n+|q/u|^n)}{|1-q^n|}
\leq\frac{n^h(a^n+b^n)}{1-r},
\qquad
\frac{n^h|q|^n}{|1-q^n|}
\leq\frac{n^hr^n}{1-r}.
\]
Both majorants are summable over $n\geq1$, and applying $D_u$ only
increases $h$.
\end{proof}

\begin{lemma}
\label[lemma]{lem:mode-trace-limit}
For a product of separately normal-ordered blocks of the fields
\eqref{eq:derived-field}, the mode sum defining $\Omega_{q,y}$ converges
absolutely and locally uniformly in \eqref{eq:nested-annuli}.  In
particular, finite K\"unneth sums and Laurent coefficient extraction on
chamber contours commute with it.
\end{lemma}

\begin{proof}
Each mode monomial is a product of finitely many modes, so the
coefficientwise trace identities of \Cref{lem:shifted-trace},
\Cref{prop:graded-wick-recursion}, and
\Cref{lem:normal-ordering-subtraction} apply to it and express its
normalized trace as a finite sum over partitions into singletons and
pairs.  Up to fixed factors depending on the cohomology classes, a
singleton is bounded by the first majorant in
\Cref{lem:normal-convergence} with $u=x_i/y$, a pair between distinct
blocks by the same majorant with $u=x_j/x_i$, and a pair within one
block, of covariance $C-C^0$, by the second majorant; the last two carry
one additional power of $n$.  The sums over the disjoint parts of a
partition therefore converge absolutely and uniformly, and there are
finitely many partitions.  K\"unneth expansions are finite because $H$
is finite dimensional, and uniform convergence on the coefficient
contours permits termwise integration.
\end{proof}

\subsection{The Wick expansion for currents}

Finally, we compute \(\Omega_{q,y}\) for operator-valued currents via the following main theorem of this section.

\begin{theorem}[Graded Wick theorem for the normalized trace]\label{thm:wick-current}
For $i=1,\ldots,N$, let $H_i$ be a finite ordered set.  Attach to every $a\in H_i$ a homogeneous class $\alpha_a\in H$ and an even integer $d_a\geq0$, and define the normally ordered block
\begin{equation}\label{eq:normal-current-block}
\mathcal O_i(x_i)
=:
\prod_{a\in \mathsf H_i}^{\longrightarrow}
\Afield^{(d_a)}(\alpha_a;x_i):.
\end{equation}
In the chamber \eqref{eq:nested-chamber},
\begin{equation}\label{eq:wick-current-formula}
\Omega_{q,y}\!\left(
\mathcal O_1(x_1)\cdots\mathcal O_N(x_N)
\right)
=
\sum_{\Pi}\eps(\Pi)
\prod_{\{a\}\in\Pi}M_a
\prod_{\substack{\{a,b\}\in\Pi\\i(a)\neq i(b)}}C_{ab}
\prod_{\substack{\{a,b\}\in\Pi\\i(a)=i(b)}}L_{ab},
\end{equation}
where $\Pi$ runs over partitions of $\bigsqcup_i\mathsf H_i$ into singletons and pairs.  The factors are
\begin{align}
M_a
&=
\left(\int_X\alpha_a\right)
D_u^{d_a}\zhat(u)\big|_{u=x_{i(a)}/y},
\label{eq:singleton-kernel}\\*
C_{ab}
&=-
(\alpha_a,\alpha_b)
D_u^{d_a+d_b}P(u)\big|_{u=x_{j}/x_i}
\quad\text{if }a\in H_i,\ b\in H_j,\ i<j,
\label{eq:cross-kernel}\\*
L_{ab}
&=-2(\alpha_a,\alpha_b)
T_{d_a+d_b+2}(q)
\quad\text{if }a,b\in \mathsf H_i.
\label{eq:loop-kernel}
\end{align}
The sign $\eps(\Pi)$ is defined in \eqref{eq:koszul-sign}, using the concatenated order $H_1,\ldots,H_N$; $i(a)$ is the index of the set containing $a$.
\end{theorem}

\begin{proof}
\emph{(i) Fixed modes.}
For $a\in H_i$, expand
\[
\Afield^{(d_a)}(\alpha_a;x_i)
=\sum_{m_a\neq0}(-m_a)^{d_a}x_i^{-m_a}
\Nak_{m_a}(\alpha_a).
\]
For a fixed assignment $(m_a)_a$, apply
\eqref{eq:normal-ordered-block-trace} with
$c_a=\Nak_{m_a}(\alpha_a)$.  Multiply by
$\prod_a(-m_a)^{d_a}x_{i(a)}^{-m_a}$ and sum over the modes.
The summation factors over the singletons and pairs of each partition.
\Cref{lem:mode-trace-limit} permits this passage from finite-mode
traces to current traces.

\emph{(ii) Singletons.}
For $a\in H_i$, put $u=x_i/y$.  The mean in
\Cref{lem:shifted-trace} gives, since $d_a$ is even,
\[
\sum_{m\neq0}(-m)^{d_a}x_i^{-m}\mu_m(\alpha_a)
=\left(\int_X\alpha_a\right)
\sum_{n>0}n^{d_a}
\left(\frac{u^n}{1-q^n}-\frac{q^nu^{-n}}{1-q^n}\right).
\]
By \eqref{eq:Zhat-Fourier}, this is $M_a$ in
\eqref{eq:singleton-kernel}.

\emph{(iii) Pairs in different blocks.}
Let $a\in H_i$, $b\in H_j$ with $i<j$, and set
$u=x_j/x_i$ and $s=d_a+d_b$.  Only opposite mode indices contribute.
Using \eqref{eq:covariance-plusminus}--\eqref{eq:covariance-minusplus},
the assignments $(m_a,m_b)=(n,-n)$ and $(-n,n)$ give, respectively,
\[
-(\alpha_a,\alpha_b)\frac{n^{s+1}u^n}{1-q^n},
\qquad
-(\alpha_a,\alpha_b)\frac{n^{s+1}q^nu^{-n}}{1-q^n}.
\]
Their sum over $n>0$ is
$-(\alpha_a,\alpha_b)D_u^sP(u)$, which is
\eqref{eq:cross-kernel}.

\emph{(iv) Pairs in one block and assembly.}
For $a,b\in H_i$, the covariance is $C-C^0$.  In either orientation,
\[
(C-C^0)(\Nak_{\pm n}(\alpha_a),\Nak_{\mp n}(\alpha_b))
=-(\alpha_a,\alpha_b)\frac{nq^n}{1-q^n}.
\]
The powers of $x_i$ cancel, and the even derivative orders multiply
each term by $n^{d_a+d_b}$.  The two orientations therefore sum to
\[
-2(\alpha_a,\alpha_b)
\sum_{n>0}n^{d_a+d_b+1}\frac{q^n}{1-q^n}
=-2(\alpha_a,\alpha_b)T_{d_a+d_b+2}(q),
\]
as in \eqref{eq:loop-kernel}.  Substituting the three evaluated sums
into \eqref{eq:normal-ordered-block-trace} proves
\eqref{eq:wick-current-formula}, with the sign \eqref{eq:koszul-sign}.
\end{proof}

\subsection{Cohomological contributions of Wick graphs}\label{subsec:wick-graphs}

\Cref{thm:wick-current} is sufficient to prove quasimodularity. However, to compute the reduced generating function explicitly, as we will do in \Cref{sec:algorithm} and \Cref{sec:checks}, we provide a more convenient formulation in \Cref{cor:basis-free-graphs} and \Cref{cor:small-diagonal-cohomological-contribution}. We first introduce Wick graphs, which are combinatorial gadgets that keep track of the summation indices in \eqref{eq:wick-current-formula}.

\begin{definition}\label[definition]{def:wick-graph}
Fix the ordered current blocks of \Cref{thm:wick-current}.  Represent the $i$-th block by a vertex $i$, and call each field occurrence $a\in H_i$ a \emph{labeled half-edge} attached to that vertex.  A \emph{Wick graph} $\Gamma$ is a partition $\Pi_\Gamma$ of $\bigsqcup_iH_i$ into singletons and pairs, together with the attachment map $a\mapsto i(a)$.  A singleton is an unpaired half-edge; a pair joining distinct vertices is an \emph{intervertex edge}, and a pair at one vertex is a \emph{loop}.

The scalar function associated with $\Gamma$ is
\[
\begin{aligned}
\mathscr F_\Gamma(z_0,\ldots,z_N;\tau)
={}&
\prod_{\{a\}\in\Pi_\Gamma}
D^{d_a}\zhat(z_{i(a)}-z_0)\\
&\cdot
\prod_{\substack{\{a,b\}\in\Pi_\Gamma\\i(a)<i(b)}}
\bigl(-D^{d_a+d_b}P(z_{i(b)}-z_{i(a)})\bigr)
\cdot
\prod_{\substack{\{a,b\}\in\Pi_\Gamma\\i(a)=i(b)}}
\bigl(-2T_{d_a+d_b+2}(q)\bigr),
\end{aligned}
\]
where $D$ differentiates the displayed difference variable.  This is the scalar part of the contribution of $\Gamma$ to \eqref{eq:wick-current-formula}; its cohomological part is defined in \Cref{def:cohomological-contribution}.
\end{definition}

While \Cref{thm:wick-current} is stated for individual cohomology labels, \Cref{cor:basis-free-graphs} combines these contributions into an expression independent of the choice of K\"unneth decomposition. Let $T_i\in H^{\otimes r_i}$ and choose a decomposition
\[
T_i=\sum_\nu\beta_{\nu,1}\otimes\cdots\otimes\beta_{\nu,r_i}
\]
with homogeneous factors.  For even derivative orders
$\mathbf d_i=(d_{i,1},\ldots,d_{i,r_i})$, define
\begin{equation}\label{eq:tensor-current}
\mathcal J(T_i,\mathbf d_i;x_i)
=\frac1{r_i!}\sum_\nu
:\prod_{h=1}^{r_i}
\Afield^{(d_{i,h})}(\beta_{\nu,h};x_i):.
\end{equation}
As in \eqref{eq:diagonal-operator}, multilinearity in the ordered
tensor factors makes this definition independent of the decomposition;
all normal ordering uses \eqref{eq:super-normal-ordering}.

\begin{definition}[Cohomological contribution]\label[definition]{def:cohomological-contribution}
Let $\Gamma$ be a Wick graph on the ordered half-edge set
$H_1\sqcup\cdots\sqcup H_N$, where $|H_i|=r_i$.  Order the parts of
$\Pi_\Gamma$ by their least half-edge and order the two elements of every
pair increasingly, as in \Cref{thm:wick-current}.  The super-permutation
that moves the tensor factors from the original half-edge order to this
order by parts defines
\[
\Sigma_\Gamma:
\bigotimes_{i=1}^N H^{\otimes r_i}
\longrightarrow
\bigotimes_{B\in\Pi_\Gamma}H^{\otimes |B|}.
\]
For a part $B\in\Pi_\Gamma$, define a linear functional
\[
\lambda_B=
\begin{cases}
\displaystyle\int_X:H\longrightarrow\Q,&B=\{a\},\\[2mm]
(\ ,\ ):H\otimes H\longrightarrow\Q,&B=\{a,b\}.
\end{cases}
\]
The \emph{cohomological contribution} of $\Gamma$ is the scalar
\begin{equation}\label{eq:cohomological-contribution}
\mathcal I_\Gamma(T_1,\ldots,T_N)
=
\left(\bigotimes_{B\in\Pi_\Gamma}\lambda_B\right)
\Sigma_\Gamma(T_1\otimes\cdots\otimes T_N).
\end{equation}
The sign of the super-permutation in \eqref{eq:cohomological-contribution}
is precisely the Koszul sign in \eqref{eq:wick-current-formula}.  If
$T_i=\tau_{r_i*}\gamma_i$, we abbreviate the left-hand side by
$\mathcal I_\Gamma(\gamma_1,\ldots,\gamma_N)$.
\end{definition}

\begin{corollary}[Graph expansion with cohomology labels]\label[corollary]{cor:basis-free-graphs}
With $x_i=e^{2\pi iz_i}$ and $y=e^{2\pi iz_0}$, one has
\begin{equation}\label{eq:tensor-graph-expansion}
\Omega_{q,y}\!\left(
\prod_{i=1}^N\mathcal J(T_i,\mathbf d_i;x_i)
\right)
=\frac1{\prod_{i=1}^Nr_i!}
\sum_\Gamma
\mathcal I_\Gamma(T_1,\ldots,T_N)
\mathscr F_\Gamma(z_0,\ldots,z_N;\tau),
\end{equation}
where the sum runs over the Wick graphs of \Cref{def:wick-graph}.
\end{corollary}

\begin{proof}
Apply \Cref{thm:wick-current} to the homogeneous tensor factors in
\eqref{eq:tensor-current}.  A singleton contributes the cohomological
factor $\int_X\alpha_a$, and either kind of pair contributes
$(\alpha_a,\alpha_b)$.  Their ordered product, with the sign
\eqref{eq:koszul-sign}, is the evaluation by $\Sigma_\Gamma$ and the
functionals $\lambda_B$ in \eqref{eq:cohomological-contribution}.
Summing over the tensor decompositions therefore gives
$\mathcal I_\Gamma(T_1,\ldots,T_N)$, while the scalar kernels give
$\mathscr F_\Gamma$.  The factors $1/r_i!$ come from
\eqref{eq:tensor-current}.  Since \eqref{eq:cohomological-contribution} is a
composition of canonical maps in the category of super-vector spaces, the
result is independent of the K\"unneth expansions.
\end{proof}

For small diagonals, pairing two tensor factors has a familiar geometric interpretation.  The remaining cohomological identities hold for every smooth projective surface, without an assumption on $K_X$.

\begin{lemma}\label[lemma]{lem:diagonal-self-intersection}
Contracting two tensor factors of $\tau_{r*}\alpha$ by the Poincar\'e pairing, after the Koszul permutation bringing those factors together, gives
\begin{equation}\label{eq:diagonal-contraction}
\operatorname{contr}_{a,b}(\tau_{r*}\alpha)
=\tau_{r-2,*}(\eX\alpha).
\end{equation}
For $r=2$, the right-hand side means $\tau_{0,*}(\eX\alpha):=\int_X\eX\alpha$.
\end{lemma}

\begin{proof}
This is the cohomological form of the diagonal self-intersection formula
\begin{equation}\label{eq:diagonal-self-intersection}
\Delta^*\Delta_*(1)=c_2(T_X)=\eX.
\end{equation}
Apply the projection formula to the two factors being contracted and leave the other $r-2$ factors unchanged.
\end{proof}

For even cohomology classes, the cohomological contribution has the
following explicit form.

\begin{corollary}\label[corollary]{cor:small-diagonal-cohomological-contribution}
Suppose that $\gamma_1,\ldots,\gamma_N$ are even and
$T_i=\tau_{r_i*}\gamma_i$.  Let $\overline\Gamma$ be the multigraph with
vertices $1,\ldots,N$ and one edge for each pair of $\Gamma$, including
loops.  For every connected component $C$ of $\overline\Gamma$, set
\[
g_C=|E(C)|-|V(C)|+1.
\]
Then
\begin{equation}\label{eq:small-diagonal-cohomological-contribution}
\mathcal I_\Gamma(\gamma_1,\ldots,\gamma_N)
=\prod_{C\in\pi_0(\overline\Gamma)}
\int_X\eX^{g_C}\prod_{i\in V(C)}\gamma_i.
\end{equation}
\end{corollary}

\begin{proof}
The classes $\gamma_i$ have even degree and therefore commute
without additional signs.  The Koszul
signs of their K\"unneth factors are already included in the graded pairings
and the diagonal-contraction identity.  In each connected component, use
the Poincar\'e pairings along a spanning tree to
replace their labels by the product $\prod_{i\in V(C)}\gamma_i$
on a single small diagonal.  Every one of
the remaining $g_C$ pairings is a self-intersection of that diagonal and
therefore inserts one factor of $\eX$ by
\Cref{lem:diagonal-self-intersection}.  After all pairings and singleton
integrations have been applied, the integral over $X$ of the resulting class
is the expression in
\eqref{eq:small-diagonal-cohomological-contribution}.
\end{proof}

\section{Quasi-elliptic constant terms and the weight bound}\label{sec:quasielliptic}

In \Cref{sec:wick}, we proved that \(\widehat F^{\boldsymbol\alpha}_{\mathbf j}(q)\) is the constant term of certain combinations of the quasi-elliptic functions \(\widehat{Z}\), \(P\), and \(T\). Here, by applying the theorem of Goujard and M\"oller \cite{GoujardMoller2020} on the quasimodularity of such constant terms, we complete the proof of \Cref{thm:G-insertions}.

\subsection{The Goujard--M\"oller constant-term theorem}

We recall only the part of the quasi-elliptic formalism needed here.  Let
$z=(z_0,\ldots,z_N)$ and $q=e^{2\pi i\tau}$.  Goujard and M\"oller define a graded ring
\[
\QEll_{N+1}=\bigoplus_{w\geq0}\QEll_{N+1}^{(w)}
\]
of quasi-elliptic quasimodular functions in the variables $z_i$~\cite[Section~5.3]{GoujardMoller2020}.  Its elements are meromorphic in the elliptic variables, have poles only where $z_i-z_j\in\Z+\tau\Z$, are periodic under $z_i\mapsto z_i+1$, and transform polynomially under $z_i\mapsto z_i+\tau$.  The coefficients of the resulting polynomials are quasimodular in $\tau$.  The ring contains
\[
Z(z_i-z_j),\qquad D_z^mP(z_i-z_j),
\]
with weights $1$ and $m+2$, respectively, and differentiation with respect to an elliptic variable raises weight by one; these are the grading conventions of ~\cite[Section~5.3]{GoujardMoller2020}.  Products are graded by the sum of weights.

For a permutation $\pi$ of $\{0,\ldots,N\}$, choose real numbers
\[
0\leq\eps_{\pi(0)}<\eps_{\pi(1)}<\cdots<\eps_{\pi(N)}<\Im\tau
\]
and define the \emph{chamber constant term} \(\text{CT}_\pi\) by integrating over the horizontal contours
$z_i=t_i+i\eps_i$, $0\leq t_i\leq1$.  Equivalently, this is the simultaneous Fourier coefficient of $x_0^0\cdots x_N^0$ in the corresponding nested annulus.

\begin{theorem}[{\cite[Theorem~5.8]{GoujardMoller2020}}]\label{thm:GM}
If $f\in\QEll_{N+1}^{(w)}$, then its constant term in every chamber is a quasimodular form of mixed weight at most $w$:
\begin{equation}\label{eq:GM-theorem}
\CT_\pi(f)\in\Fil_{\leq w}\QM.
\end{equation}
The same conclusion holds for a function of mixed quasi-elliptic weight at most $w$.
\end{theorem}

\begin{remark}[Scope of the constant-term theorem]\label[remark]{rem:GM-full-generality}
The point needed in the proof of \Cref{thm:G-insertions} is that \Cref{thm:GM} applies to the full ring $\QEll_{N+1}$ and to every chamber.  By \eqref{eq:ZP-identification} and \Cref{def:wick-graph}, every factor of $\mathscr F_\Gamma$ is a derivative of $Z$ or $P$, or a quasimodular coefficient; hence each associated Wick-graph function lies in this ring.  Therefore \Cref{thm:GM} applies to the chamber \eqref{eq:nested-chamber} used by the operator trace.  This remark records precisely the scope needed to pass from the Wick formula \eqref{eq:wick-current-formula} to quasimodularity.
\end{remark}

\subsection{Weights of Wick-graph functions}

Assign to a half-edge carrying derivative order $d$ the weight
\begin{equation}\label{eq:half-edge-weight}
\omega(d)=1+d.
\end{equation}

\begin{proposition}[Weight bound for associated functions]\label[proposition]{prop:graph-weight}
Let $\Gamma$ be a Wick graph and let $\mathscr F_\Gamma$ be its associated function in the sense of \Cref{def:wick-graph}.  Then
\begin{equation}\label{eq:graph-function-weight}
\mathscr F_\Gamma
\in
\bigoplus_{0\leq v\leq W_\Gamma}\QEll_{N+1}^{(v)},
\qquad
W_\Gamma=\sum_{a\in\bigsqcup_i\mathsf H_i}(1+d_a).
\end{equation}
More precisely, the factors defining $\mathscr F_\Gamma$ satisfy the following bounds:
\begin{enumerate}[label=\textup{(\roman*)},leftmargin=28pt]
\item the singleton factor $D^d\zhat$ has mixed weight at most $1+d$;
\item the intervertex factor $D^{d+e}P$ is homogeneous of weight $2+d+e$;
\item the loop factor $T_{d+e+2}$ has mixed weight at most $2+d+e$.
\end{enumerate}
\end{proposition}

\begin{proof}
The function $Z$ is homogeneous of weight $1$ in $\QEll$, so $\zhat=Z-\frac12$ has components of weights $1$ and $0$; because elliptic differentiation raises weight by one, $D^d\zhat$ has mixed weight at most $1+d$.  Since $P=DZ$ has weight $2$, the factor $D^{d+e}P$ has weight $2+d+e$.  Finally, all derivative orders are even, and hence \eqref{eq:loop-series-Eisenstein} writes $T_{d+e+2}$ as a sum of weights $0$ and $d+e+2$.  Multiplication adds quasi-elliptic weights, while every half-edge belongs to exactly one singleton or pair of $\Pi_\Gamma$; adding these factor bounds gives \eqref{eq:graph-function-weight}.
\end{proof}

\begin{corollary}\label[corollary]{cor:current-weight}
In a Wick-graph function obtained from a product of the currents \eqref{eq:V-current} and \eqref{eq:S-current}, the half-edges at a vertex representing $\Vcur_r$ contribute $r$ to the bound $W_\Gamma$, whereas those at a vertex representing one summand of $\Scur_r$ contribute $r+2$.  Consequently, if the vertex $i$ is chosen from one of the three current terms representing $\GG_{k_i}(\alpha_i)$ in \eqref{eq:current-reduction}, then every associated function satisfies
\begin{equation}\label{eq:current-graph-function-weight}
\mathscr F_\Gamma
\in
\bigoplus_{0\leq v\leq\sum_i(k_i+2)}
\QEll_{N+1}^{(v)}.
\end{equation}
\end{corollary}

\begin{proof}
Every half-edge of $\Vcur_r$ has derivative order $0$, so its contribution to $W_\Gamma$ is $r$.  Each summand of $\Scur_r$ has one half-edge of derivative order $2$ and $r-1$ of order $0$, giving
\[
(1+2)+(r-1)(1+0)=r+2.
\]
At a vertex representing $\GG_k(\alpha)$, the three possibilities in \eqref{eq:current-reduction} therefore contribute, respectively,
\[
k+2,
\qquad
k+2,
\qquad
k.
\]
Summing over vertices and applying \Cref{prop:graph-weight} proves \eqref{eq:current-graph-function-weight}.  All loop indices are even; indeed, a loop in one summand of $\Scur$ has at most one endpoint with derivative order $2$ by \eqref{eq:S-marked-half-edge}, so only $T_2$ and $T_4$ occur.
\end{proof}

\subsection{From operator constant terms to chamber constant terms}

We use the notation $x_0=y$ introduced with the scalar kernels and write $\Omega_{q,x_0}$ for the normalized trace.

Define the operator-valued series
\begin{equation}\label{eq:C-current}
\Ccur_k(\alpha;x)
=-\Vcur_{k+2}(\alpha;x)
+\frac1{24}
\left(
\Scur_k(\eX\alpha;x)-2\Vcur_k(\eX\alpha;x)
\right).
\end{equation}
By \Cref{prop:current-reduction},
\begin{equation}\label{eq:G-is-CT-C}
\GG_k(\alpha)=[x^0]\Ccur_k(\alpha;x).
\end{equation}
For fixed pairs $(j_i,\alpha_i)$, put
\begin{equation}\label{eq:Phi-def}
\Phi(x_0,\ldots,x_N;q)
=
\Omega_{q,x_0}\!\left(
\prod_{i=1}^N\Ccur_{j_i}(\alpha_i;x_i)
\right).
\end{equation}

\begin{lemma}\label[lemma]{lem:trace-CT-interchange}
In the chamber \eqref{eq:nested-chamber},
\begin{equation}\label{eq:partial-CT-Phi}
\Omega_{q,x_0}\!\left(
\prod_{i=1}^N\GG_{j_i}(\alpha_i)
\right)
=
[x_1^0\cdots x_N^0]\Phi(x_0,\ldots,x_N;q).
\end{equation}
Moreover, every Laurent monomial in $\Phi$ has total exponent zero, and hence
\begin{equation}\label{eq:full-CT-Phi}
[x_1^0\cdots x_N^0]\Phi(1,x_1,\ldots,x_N;q)
=
[x_0^0x_1^0\cdots x_N^0]\Phi(x_0,\ldots,x_N;q).
\end{equation}
The right-hand side is the chamber constant term to which \Cref{thm:GM} applies.
\end{lemma}

\begin{proof}
For $N_0\in\mathbb N$, truncate every field in $\Ccur_{j_i}$ to
$0<|m|\leq N_0$, and denote the resulting current by
$\Ccur_{j_i,N_0}$, with constant term
$\GG_{j_i,N_0}:=[x_i^0]\Ccur_{j_i,N_0}(\alpha_i;x_i)$.
A normally ordered zero-mode monomial has
\[
p=\sum_{m_a>0}m_a=\sum_{m_a<0}(-m_a);
\]
its annihilation operators act first and kill $H^*(X^{[n]})$ if $p>n$,
and otherwise every $|m_a|\leq n$.  Hence $\GG_{j_i,N_0}$ agrees with
$\GG_{j_i}(\alpha_i)$ on $H^*(X^{[n]})$ once $N_0\geq n$.  Both preserve
the number grading, so their products do as well, and the normalized
traces of $\prod_i\GG_{j_i,N_0}$ stabilize through $q^{N_0}$.
At finite cutoff, linearity gives
\[
\Omega_{q,x_0}\!\left(\prod_i\GG_{j_i,N_0}\right)
=[x_1^0\cdots x_N^0]\,
\Omega_{q,x_0}\!\left(\prod_i\Ccur_{j_i,N_0}(\alpha_i;x_i)\right).
\]
By \Cref{lem:mode-trace-limit} the limit on the right passes through the
coefficient contours, while the left side stabilizes coefficientwise to
the trace of $\prod_i\GG_{j_i}(\alpha_i)$; uniform convergence on fixed
chamber contours and a compatible circle in the $q$-plane identifies
these stabilized coefficients with those of the analytic constant term.
This proves \eqref{eq:partial-CT-Phi}, with all products kept in the
order $1,\ldots,N$.

Every scalar graph factor in \Cref{def:wick-graph} depends only on a ratio $x_i/x_0$ or $x_j/x_i$, while a loop factor is independent of all $x$-variables.  Therefore each Laurent monomial has total exponent zero.  If the exponents of $x_1,\ldots,x_N$ are all zero, the exponent of $x_0$ is automatically zero, proving \eqref{eq:full-CT-Phi}.  The contours specified by \eqref{eq:nested-chamber} are exactly the chamber contours in \Cref{thm:GM}.
\end{proof}

We can now prove \Cref{thm:G-insertions}, stated in \Cref{sec:introduction}.

\begin{proof}[Proof of \Cref{thm:G-insertions}]
By \Cref{prop:reduced-trace}, the left-hand side of \eqref{eq:G-insertion-result} is
\[
\Omega_{q,y}\!\left(
\prod_{i=1}^N\GG_{j_i}(\alpha_i)
\right).
\]
By \Cref{lem:trace-CT-interchange}, this is the full chamber constant term in \eqref{eq:full-CT-Phi}.  \Cref{cor:basis-free-graphs} expresses $\Phi$ as a finite sum of the associated functions $\mathscr F_\Gamma$ from \Cref{def:wick-graph}, multiplied by the $q$-independent cohomological contributions $\mathcal I_\Gamma$ of \Cref{def:cohomological-contribution}.  By \Cref{cor:current-weight}, each $\mathscr F_\Gamma$ lies in the mixed quasi-elliptic subspace of weight at most $\sum_i(j_i+2)$.  Apply \Cref{thm:GM} to each associated function and sum.  The resulting Fourier coefficients are rational: the scalar factors in \eqref{eq:singleton-kernel}--\eqref{eq:loop-kernel} have rational Fourier coefficients, and every $\mathcal I_\Gamma$ is a rational number by \eqref{eq:cohomological-contribution}.  As the filtered space of level-one quasimodular forms has a basis with rational Fourier coefficients, the expression belongs to $\Q[E_2,E_4,E_6]$.
\end{proof}

\section{Proof of Qin's conjecture}\label{sec:main-proof}

We return from the universal classes $G_k(\alpha,n)$ to tautological Chern characters.  The required identity is the GRR decomposition in \Cref{prop:GRR-decomposition}.

\begin{proposition}[GRR decomposition]\label[proposition]{prop:GRR-decomposition}
Let $L$ be a line bundle on $X$.  For every $k\geq0$,
\begin{equation}\label{eq:GRR-decomposition}
\ch_k(L^{[n]})
=
G_k(\1,n)
+G_{k-1}(c_1(L),n)
+G_{k-2}\!\left(\frac{c_1(L)^2}{2},n\right),
\end{equation}
where $G_j=0$ for $j<0$.
\end{proposition}

\begin{proof}
The projection $\mathcal Z_n\to X^{[n]}$ is finite and flat of degree $n$, and
\[
L^{[n]}=p_{1*}(\mathcal O_{\mathcal Z_n}\otimes p_2^*L).
\]
Grothendieck--Riemann--Roch gives
\begin{equation}
\ch(L^{[n]})
=p_{1*}\!\left(
\ch(\mathcal O_{\mathcal Z_n})
 p_2^*\ch(L)
 p_2^*\td(X)
\right)
=G(\ch(L),n).
\label{eq:GRR-total}
\end{equation}
On a surface,
$\ch(L)=1+c_1(L)+\frac12c_1(L)^2$.  Taking the component of degree $2k$ in \eqref{eq:GRR-total} proves \eqref{eq:GRR-decomposition}.  The same identity is used in~\cite{QinYu2018,AlhwaimelQin2024}.
\end{proof}

\begin{proof}[Proof of \Cref{thm:main}]
Put \(\ell_i=c_1(L_i)\) and
\[
W=\sum_{i=1}^N(k_i+2).
\]
We make the reduction to \Cref{thm:G-insertions} explicit.

\begin{enumerate}[label=\textup{(\arabic*)},ref=\textup{\arabic*},leftmargin=34pt]
\item\label{main:GRR-expand}
For \(\varepsilon\in\{0,1,2\}\), define
\[
j_i(\varepsilon)=k_i-\varepsilon,
\qquad
\alpha_i(\varepsilon)=
\begin{cases}
\1,&\varepsilon=0,\\
\ell_i,&\varepsilon=1,\\
\ell_i^2/2,&\varepsilon=2.
\end{cases}
\]
Call a choice \emph{admissible} if \(j_i(\varepsilon)\geq0\).  The GRR identity
\eqref{eq:GRR-decomposition} is precisely
\begin{equation}\label{eq:main-GRR-sum}
\ch_{k_i}(L_i^{[n]})
=
\sum_{\substack{\varepsilon_i\in\{0,1,2\}\\
j_i(\varepsilon_i)\geq0}}
G_{j_i(\varepsilon_i)}
 \bigl(\alpha_i(\varepsilon_i),n\bigr).
\end{equation}

\item\label{main:product-expand}
Multiply \eqref{eq:main-GRR-sum} for \(i=1,\ldots,N\), integrate against
\(c(T_{X^{[n]}})\), and sum over \(n\).  If
\(\boldsymbol\varepsilon=(\varepsilon_1,\ldots,\varepsilon_N)\), this gives
\begin{align}
\left\langle\prod_{i=1}^N\ch_{k_i}^{L_i}\right\rangle'
={}&
\sum_{\boldsymbol\varepsilon\ \text{admissible}}
\qpoch^{\chi(X)}
F_{\mathbf j(\boldsymbol\varepsilon)}
 ^{\boldsymbol\alpha(\boldsymbol\varepsilon)}(q),
\label{eq:main-universal-expansion}\\
F_{\mathbf j}^{\boldsymbol\alpha}(q)
:={}&
\sum_{n\geq0}q^n
\int_{X^{[n]}}
\left(\prod_{i=1}^NG_{j_i}(\alpha_i,n)\right)
c(T_{X^{[n]}}).
\notag
\end{align}
There are at most \(3^N\) summands, so the expansion is finite.  Notice that
the factor \(1/2\) in the third GRR branch is already included in
\(\alpha_i(2)=\ell_i^2/2\).

\item\label{main:apply-universal}
Apply \Cref{thm:G-insertions} separately to every summand in
\eqref{eq:main-universal-expansion}.  For a fixed admissible branch choice
$\boldsymbol\varepsilon$,
\begin{equation}\label{eq:main-branch-bound}
\qpoch^{\chi(X)}
F_{\mathbf j(\boldsymbol\varepsilon)}
 ^{\boldsymbol\alpha(\boldsymbol\varepsilon)}(q)
\in
\Fil_{\leq
\sum_{i=1}^N(j_i(\varepsilon_i)+2)}\QM.
\end{equation}

\item\label{main:compare-weights}
The contribution of the \(i\)-th branch to the bound in
\eqref{eq:main-branch-bound} is
\[
j_i(\varepsilon_i)+2
=k_i+2-\varepsilon_i
=
\begin{cases}
k_i+2,&\varepsilon_i=0,\\
k_i+1,&\varepsilon_i=1,\\
k_i,&\varepsilon_i=2.
\end{cases}
\]
Consequently every admissible branch choice has total mixed weight at most
\(W\).

\item\label{main:finite-sum}
The filtration \(\Fil_{\leq W}\QM\) is a \(\Q\)-vector space.  Therefore the
finite sum in \eqref{eq:main-universal-expansion} also lies in
\(\Fil_{\leq W}\QM\), which is exactly \eqref{eq:main-theorem}.
\end{enumerate}
\end{proof}

\section{Computing the correlation series}\label{sec:algorithm}

This section explains how \Cref{sec:operators,sec:wick,sec:quasielliptic,sec:main-proof} combine into an actual computation.  We first describe the complete chain of reductions, and then isolate a finite algorithm.  No proof from the earlier sections is repeated, but every operation used in the computation is specified.

\subsection{The quantity to be computed}

Assume throughout this section that $K_X$ is numerically trivial.  Fix an ordered list of homogeneous classes
$\boldsymbol\alpha=(\alpha_1,\ldots,\alpha_N)$ in $H=H^*(X,\Q)$ and nonnegative integers
$\mathbf j=(j_1,\ldots,j_N)$.  The corresponding correlation series is
\begin{equation}\label{eq:algorithm-F-function}
F_{\mathbf j}^{\boldsymbol\alpha}(q)
=\sum_{n\geq0}q^n
\int_{X^{[n]}}
\left(\prod_{i=1}^NG_{j_i}(\alpha_i,n)\right)
c(T_{X^{[n]}}).
\end{equation}
The order of these factors is fixed; this determines the Koszul signs when odd classes occur and also fixes the nested chamber \eqref{eq:algorithm-nested-chamber}.  Define the reduced series
\begin{equation}\label{eq:algorithm-reduced-F}
\widehat F_{\mathbf j}^{\boldsymbol\alpha}(q)
:=\qpoch^{\chi(X)}F_{\mathbf j}^{\boldsymbol\alpha}(q),
\qquad
W_F:=\sum_{i=1}^N(j_i+2).
\end{equation}
\Cref{thm:G-insertions} states that
\begin{equation}\label{eq:algorithm-F-output}
\widehat F_{\mathbf j}^{\boldsymbol\alpha}(q)
\in\Fil_{\leq W_F}\Q[E_2,E_4,E_6].
\end{equation}
Thus the finite graph calculation naturally produces $\widehat F$; the unreduced series \eqref{eq:algorithm-F-function} is recovered by multiplying by $\qpoch^{-\chi(X)}$.

The assumption on $K_X$ is used both in the current reduction
\eqref{eq:current-reduction} and in the simplified Carlsson--Okounkov
operator underlying the scalar kernels, as well as in the quasimodularity
statement \eqref{eq:algorithm-F-output}.  The intersection formula
\eqref{eq:small-diagonal-cohomological-contribution} is valid without
this assumption, but this alone does not extend the complete computation
below to arbitrary surfaces.  The rank example in
\Cref{prop:ch0-check} uses the same hypothesis and checks the coefficients
of the algorithm against the known rank series.

For tautological line bundles, the original quantity in Qin's conjecture is a finite sum of these reduced correlation series.  Indeed, if $\ell_i=c_1(L_i)$, then \Cref{prop:GRR-decomposition} gives
\begin{equation}\label{eq:algorithm-GRR-wrapper}
\left\langle\prod_{i=1}^N\ch_{k_i}^{L_i}\right\rangle'
=
\sum_{\boldsymbol\varepsilon\ \text{admissible}}
\widehat F_{\mathbf j(\boldsymbol\varepsilon)}
 ^{\boldsymbol\alpha(\boldsymbol\varepsilon)}(q),
\end{equation}
with $j_i(\varepsilon_i)$ and $\alpha_i(\varepsilon_i)$ defined in
\eqref{eq:main-GRR-sum}.  Consequently, the computation of Qin's series reduces to the correlation series defined in \eqref{eq:algorithm-F-function}.

\subsection{How the earlier ingredients fit together}\label{subsec:computational-reduction}

We now summarize the mathematical passage from \eqref{eq:algorithm-F-function} to a finite graph sum.

\begin{enumerate}[label=\textup{(\arabic*)},ref=\textup{\arabic*},leftmargin=34pt]
\item\emph{From the integral to a normalized supertrace.}
\Cref{prop:reduced-trace} identifies the reduced integral series with
the normalized categorical supertrace:
\begin{equation}\label{eq:algorithm-F-as-Omega}
\widehat F_{\mathbf j}^{\boldsymbol\alpha}(q)
=\Omega_{q,y}\!\left(
\prod_{i=1}^N\GG_{j_i}(\alpha_i)
\right).
\end{equation}

\item\label{alg:ingredient-currents}\emph{From cup-product operators to operator-valued series.}
By \Cref{prop:current-reduction}, each operator in \eqref{eq:algorithm-F-as-Omega} is the constant term of
\begin{equation}\label{eq:algorithm-current-summary}
\Ccur_j(\alpha;x)
=-\Vcur_{j+2}(\alpha;x)
+\frac1{24}\Scur_j(\eX\alpha;x)
-\frac1{12}\Vcur_j(\eX\alpha;x).
\end{equation}
The current $\Vcur_r(\gamma;x)$ is $1/r!$ times a normally ordered product of $r$ Heisenberg fields, evaluated on the tensor $\tau_{r*}\gamma$ as in \eqref{eq:V-current}.  Each summand of $\Scur_r$ has $D_x^2$ applied to exactly one of these $r$ fields.  Hence a term selected from \eqref{eq:algorithm-current-summary} is completely described by a \emph{current datum}
\begin{equation}\label{eq:current-datum}
v=(r,\gamma,\mathbf d,c).
\end{equation}
Here $r\geq1$ is the number of Heisenberg fields, equivalently the number of labeled half-edges at the corresponding graph vertex; $\gamma\in H$ is the class inserted into the small diagonal $\tau_{r*}\gamma$; $\mathbf d=(d_1,\ldots,d_r)\in\{0,2\}^r$ records the derivative order on each field; and $c\in\Q$.  In the notation of \eqref{eq:tensor-current}, this datum specifies the term
\begin{equation}\label{eq:current-datum-term}
c\,\mathcal J(\tau_{r*}\gamma,\mathbf d;x)
=\frac{c}{r!}
:\prod_{h=1}^r D_x^{d_h}\Afield(x):\,(\tau_{r*}\gamma).
\end{equation}
Thus the coefficient of the displayed normally ordered product is
$c/r!$.  When referring to the class attached to a current or its graph
vertex, we mean precisely $\gamma$ in \eqref{eq:current-datum-term}.

By \Cref{lem:trace-CT-interchange}, if
\[
\Phi_{\mathbf j}^{\boldsymbol\alpha}(x_0,\ldots,x_N;q)
:=\Omega_{q,x_0}\!\left(
\prod_{i=1}^N\Ccur_{j_i}(\alpha_i;x_i)
\right),
\]
then
\begin{equation}\label{eq:algorithm-Phi-CT}
\widehat F_{\mathbf j}^{\boldsymbol\alpha}(q)
=[x_1^0\cdots x_N^0]
\Phi_{\mathbf j}^{\boldsymbol\alpha}(1,x_1,\ldots,x_N;q).
\end{equation}

\item\emph{Wick expansion of the selected currents.}
The shifted trace relation \eqref{eq:shifted-trace} gives the normalized trace of one Nakajima operator in \eqref{eq:mode-mean}.  After this scalar part is subtracted, \Cref{prop:graded-wick-recursion} shows that a product of an odd number of centered operators has zero normalized trace, while a product of an even number is a signed sum over pairings.  Normal ordering changes only pairs whose two operators belong to the same current; by \Cref{lem:normal-ordering-subtraction}, their two-point trace is replaced by $C-C^0$.  Thus \eqref{eq:wick-current-formula} expands a product of selected currents into a finite sum indexed by partitions of their labeled half-edges into singletons and pairs.

\item\emph{Scalar kernels and graph organization.}
\Cref{cor:basis-free-graphs} is the main computational result used below:
it evaluates the normalized trace of the selected currents as a finite
sum of cohomological contributions $\mathcal I_\Gamma$ times associated
functions $\mathscr F_\Gamma$, with coefficient $\prod_i c_i/r_i!$.
Write
\[
K_\Gamma(x_0,\ldots,x_N;q)
=\mathscr F_\Gamma(z_0,\ldots,z_N;\tau),
\qquad x_i=e^{2\pi iz_i}.
\]
This is the associated function used in \Cref{alg:F-kernels}; by
\Cref{def:wick-graph}, it is built only from $\zhat$, $P$, $T_s$,
and derivatives of $\zhat$ and $P$, through the factors
\begin{equation}\label{eq:algorithm-three-kernels}
D^d\zhat(x_i/x_0),
\qquad
-D^{d+e}P(x_j/x_i),
\qquad
-2T_{d+e+2}(q).
\end{equation}
These correspond to singletons, intervertex pairs, and loops,
respectively; \Cref{alg:F-sum} combines their constant terms with
$\mathcal I_\Gamma$ and the current coefficients.

\item\emph{Constant term and final evaluation.}
The scalar kernel product is expanded in the nested chamber
\begin{equation}\label{eq:algorithm-nested-chamber}
|q|<|x_j/x_i|<1\qquad(0\leq i<j\leq N),
\quad x_0=y.
\end{equation}
Taking the simultaneous constant term in $x_1,\ldots,x_N$ implements the zero-mode extraction in every current and therefore produces \eqref{eq:algorithm-Phi-CT}.  The result can be evaluated directly from the Fourier series \eqref{eq:Zhat-Fourier}--\eqref{eq:P-Fourier}.  Alternatively, \Cref{thm:GM} shows that the chamber constant term of every associated quasi-elliptic graph function is quasimodular, while \Cref{cor:current-weight} supplies the bound $W_F$.  This is the step that permits reconstruction from finitely many Fourier coefficients.
\end{enumerate}

\subsection{A finite algorithm for one reduced correlation series}\label{subsec:finite-algorithm}

The input consists of the graded $\Q$-algebra $H^*(X,\Q)$ with its unit $\1$, Euler class $\eX$, Poincar\'e pairing, and integration map, together with the ordered pairs $(j_i,\alpha_i)$ from \eqref{eq:algorithm-F-function}.  If an expansion in $q$ to finite order is desired, also choose $Q\geq0$ and work modulo $q^{Q+1}$.  The output is the reduced correlation series $\widehat F_{\mathbf j}^{\boldsymbol\alpha}$, either to that order or as a polynomial in $E_2,E_4,E_6$.  The procedure has Steps~F1--F5.  The tautological reduction in \Cref{subsec:tautological-wrapper} has separate labels T1--T2.

\begin{correlationsteps}
\item\label{alg:F-current-data}
\emph{Enumerate the current data.}
For each input pair $(j_i,\alpha_i)$, form the finite set
\begin{equation}\label{eq:algorithm-current-data-set}
\mathcal R(j_i,\alpha_i)=
\left\{
\begin{array}{ll}
A=(j_i+2,\alpha_i,\mathbf0,-1),&\\
B_h=(j_i,\eX\alpha_i,2\mathbf e_h,1/24),&1\leq h\leq j_i,\\
C=(j_i,\eX\alpha_i,\mathbf0,-1/12),&j_i\geq1.
\end{array}
\right.
\end{equation}
Here $\mathbf0\in\Z^{r}$ is the zero vector of the relevant length and $\mathbf e_h$ is its $h$-th standard basis vector.  These are precisely the terms of item~(\ref{alg:ingredient-currents}) in \Cref{subsec:computational-reduction}: substituting each datum into \eqref{eq:current-datum-term} gives
\[
\begin{array}{c|c}
\text{datum }v&c\,\mathcal J(\tau_{r*}\gamma,\mathbf d;x_i)\\ \hline
A&-\Vcur_{j_i+2}(\alpha_i;x_i)\\[1mm]
B_h&\displaystyle\frac1{24}
\mathcal J(\tau_{j_i*}(\eX\alpha_i),2\mathbf e_h;x_i)\\[1mm]
C&\displaystyle-\frac1{12}\Vcur_{j_i}(\eX\alpha_i;x_i).
\end{array}
\]
In particular, each $B_h$ specifies which one of the $j_i$ fields is
differentiated twice; the sum over $h$, rather than an individual
$B_h$, gives the $\Scur_{j_i}$ term:
\[
\frac1{24}\sum_{h=1}^{j_i}
\mathcal J(\tau_{j_i*}(\eX\alpha_i),2\mathbf e_h;x_i)
=\frac1{24}\Scur_{j_i}(\eX\alpha_i;x_i).
\]
Equivalently,
\begin{equation}\label{eq:algorithm-current-data-expansion}
\Ccur_{j_i}(\alpha_i;x_i)
=\sum_{v=(r,\gamma,\mathbf d,c)\in\mathcal R(j_i,\alpha_i)}
c\,\mathcal J(\tau_{r*}\gamma,\mathbf d;x_i).
\end{equation}
The $B_h$ and $C$ data are absent for $j_i=0$, in accordance with
$\Vcur_0=\Scur_0=0$; their formulas above apply only for $j_i\geq1$.
Discard a datum if its class $\gamma$ is zero in $H^*(X)$.  A \emph{current selection} is an ordered tuple
$\mathbf v=(v_1,\ldots,v_N)$ with $v_i\in\mathcal R(j_i,\alpha_i)$.  If
$v_i=(r_i,\gamma_i,\mathbf d_i,c_i)$, then $r_i$ is now defined to be the number of labeled half-edges at vertex $i$.

\item\label{alg:F-graphs}
\emph{Enumerate the Wick graphs and compute their cohomological contributions.}
For a fixed current selection, set
\[
\mathsf H_i=\{(i,h):1\leq h\leq r_i\},
\qquad
\mathsf H=\coprod_{i=1}^N\mathsf H_i.
\]
Enumerate every partition $\Gamma$ of $\mathsf H$ into parts of size one or two.  The elements of $\mathsf H_i$ remain labeled.  A one-element part is a singleton; a two-element part joining $\mathsf H_i$ and $\mathsf H_j$ with $i\neq j$ is an intervertex edge; and a two-element part contained in $\mathsf H_i$ is a loop at vertex $i$.  The symbol $\mathsf H$ here denotes a finite set, distinct from the cohomology algebra $H$.

Set $T_i=\tau_{r_i*}\gamma_i$ and compute $\mathcal I_\Gamma$ from
\eqref{eq:cohomological-contribution}.  This formula includes the Koszul sign
and applies to both even and odd cohomology classes.

When all $\gamma_i$ are even, compute the same scalar directly from
\Cref{cor:small-diagonal-cohomological-contribution}:
\begin{equation}\label{eq:algorithm-cohomological-contribution}
\mathcal I_\Gamma
=\prod_{C\in\pi_0(\overline\Gamma)}
\int_X\eX^{g_C}\prod_{i\in V(C)}\gamma_i.
\end{equation}
For each connected component, compute
$g_C=|E(C)|-|V(C)|+1$, multiply its classes $\gamma_i$ and $g_C$ copies
of $\eX$ in $H^*(X,\Q)$, and integrate the result.  The product of these
numbers is \eqref{eq:algorithm-cohomological-contribution}.  Discard
$\Gamma$ if any factor vanishes.

The degree condition gives a quick preliminary test: a connected
component $C$ can contribute only if
\begin{equation}\label{eq:algorithm-component-degree}
4g_C+\sum_{i\in V(C)}\deg\gamma_i=4.
\end{equation}
For example, if all its classes are $\1$, this requires $g_C=1$.
Satisfying the condition alone does not guarantee that the corresponding
intersection number is nonzero; that number must still be evaluated.

\item\label{alg:F-kernels}
\emph{Attach the scalar kernels and extract the chamber constant term.}
For $a=(i,h)\in\mathsf H_i$, write $d_a=d_{i,h}$ and set $x_0=y$.  Attach the following to each part of $\Gamma$:
\[
\begin{array}{c|c}
\text{part}&\text{scalar factor}\\ \hline
\{a\},\ a\in\mathsf H_i&D^{d_a}\zhat(x_i/x_0)\\
\{a,b\},\ a\in\mathsf H_i,\ b\in\mathsf H_j,\ i<j
&-D^{d_a+d_b}P(x_j/x_i)\\
\{a,b\}\subset\mathsf H_i&-2T_{d_a+d_b+2}(q).
\end{array}
\]
Let $K_\Gamma(x_0,\ldots,x_N;q)$ be the product of these factors and expand it in the nested chamber \eqref{eq:algorithm-nested-chamber}.  Define
\begin{equation}\label{eq:algorithm-scalar-CT}
\CT_\Gamma(q)
=[x_1^0\cdots x_N^0]
K_\Gamma(1,x_1,\ldots,x_N;q).
\end{equation}
The use of labeled half-edges is important: the coefficient $c_i/r_i!$
from \eqref{eq:current-datum-term} is included once for each current,
while a labeled loop contributes $-2T_s$.  There is no additional
division by the number of graph automorphisms.

\item\label{alg:F-sum}
\emph{Sum the graph contributions.}
The contribution of a current selection $\mathbf v$ and a Wick graph $\Gamma$ is
\begin{equation}\label{eq:algorithm-contribution}
\operatorname{Cont}(\mathbf v,\Gamma)
=\left(\prod_{i=1}^N\frac{c_i}{r_i!}\right)
\mathcal I_\Gamma(\gamma_1,\ldots,\gamma_N)
\CT_\Gamma(q).
\end{equation}
Sum over all current selections and all retained Wick graphs:
\begin{equation}\label{eq:algorithm-F-sum}
\widehat F_{\mathbf j}^{\boldsymbol\alpha}(q)
=\sum_{\mathbf v}
\sum_{\Gamma}
\operatorname{Cont}(\mathbf v,\Gamma).
\end{equation}
Both sums are finite.  Equation \eqref{eq:algorithm-F-sum} is the desired reduced correlation series; multiply by $\qpoch^{-\chi(X)}$ only if the unreduced series in \eqref{eq:algorithm-F-function} is required.

\item\label{alg:F-reconstruct}
\emph{Evaluate the scalar constant terms.}
For an expansion modulo $q^{Q+1}$, keep only mode indices $|n|\leq Q$ in $\zhat$ and $P$, and summation indices $n\leq Q$ in $T_s$.  Truncate the coefficients at order $Q$, multiply the resulting Laurent polynomials, and retain exponent $(0,\ldots,0)$.

The cutoff also applies to the positive modes, whose coefficients can have nonzero constant term in $q$.  To see this, represent a singleton at vertex $i$ by an edge from $0$ to $i$, and orient every intervertex edge from the smaller to the larger vertex.  A mode $n$ gives the factor $(x_j/x_i)^n$ on such an edge.  Constant-term extraction imposes conservation of these signed modes at every vertex; conservation at vertex $0$ follows because each monomial has total exponent zero.  Across any cut $\{0,\ldots,k\}$, the total positive mode equals the total absolute negative mode.  Each negative mode of size $n$ costs at least $q^n$ by \eqref{eq:Zhat-Fourier}--\eqref{eq:P-Fourier}.  Thus, in a term of $q$-degree at most $Q$, their total size is at most $Q$.  Every positive edge crosses a cut, so its mode is at most $Q$ as well.  Finally, each summand of a loop series $T_s$ starts at $q^n$.  This proves that the prescribed extraction is finite.

For a closed quasimodular expression, use the finite monomial set
\begin{equation}\label{eq:algorithm-QM-basis}
\mathcal M_{W_F}=
\left\{E_2^aE_4^bE_6^c:
a,b,c\geq0,\ 2a+4b+6c\leq W_F\right\}.
\end{equation}
Increase $Q$ until the Fourier-coefficient matrix of $\mathcal M_{W_F}$ has full column rank, solve the resulting linear system, and verify the solution at several additional orders.  Membership in \eqref{eq:algorithm-F-output} guarantees that this reconstruction terminates.
\end{correlationsteps}

\subsection{Reduction for tautological line bundles}\label{subsec:tautological-wrapper}

For completeness, we state separately how the five steps of
\Cref{subsec:finite-algorithm} compute Qin's original series.  The two operations below are labeled T1--T2, to distinguish them from the correlation steps F1--F5.  This is
exactly the GRR reduction \eqref{eq:main-universal-expansion} in the proof
of \Cref{thm:main}.

\begin{tautologicalsteps}
\item\label{alg:T-GRR}
For each pair $(k_i,L_i)$, put $\ell_i=c_1(L_i)$ and form the set of admissible GRR branches
\[
\mathcal B_i=
\left\{
(k_i,\1),\ (k_i-1,\ell_i),\
\left(k_i-2,\frac{\ell_i^2}{2}\right)
\right\},
\]
omitting a pair whose first entry is $<0$.  A \emph{GRR branch choice} (the term used throughout for this object) is an ordered tuple
\[
\mathbf b=((j_1,\alpha_1),\ldots,(j_N,\alpha_N))
\in\prod_{i=1}^N\mathcal B_i.
\]

\item\label{alg:T-sum}
For every GRR branch choice $\mathbf b$, compute
$\widehat F_{\mathbf j}^{\boldsymbol\alpha}$ by the five steps of
\Cref{subsec:finite-algorithm} and add the results.  By
\eqref{eq:algorithm-GRR-wrapper}, the sum is
\begin{equation}\label{eq:algorithm-output}
\mathcal Q_{\mathbf k,\mathbf L}^X(q)
:=
\left\langle\prod_{i=1}^N\ch_{k_i}^{L_i}\right\rangle'
\in
\Fil_{\leq\sum_i(k_i+2)}\Q[E_2,E_4,E_6].
\end{equation}
For tautological line-bundle factors the numerical input from $X$ reduces to
\[
\chi(X)=\int_X\eX,
\qquad
I_{ij}=\int_X\ell_i\ell_j.
\]
Indeed, $\eX\ell_i=\eX\ell_i^2=0$ on a surface, so the $B_h$ and $C$
terms can be nonzero only when the selected GRR class is $\alpha_i=\1$;
the class $\gamma_i$ of either correction term is then $\eX$.
For arbitrary classes $G_j(\alpha,n)$, use the product, Poincar\'e pairing,
unit, Euler class, and integration map on $H^*(X,\Q)$ specified as input
to \Cref{subsec:finite-algorithm}.
\end{tautologicalsteps}

\section{Examples: low-degree calculations}\label{sec:checks}

We compute the rank term, a one-point series, and a two-point series.
The examples illustrate the three parts of the procedure in
\Cref{sec:algorithm}: selecting the nonzero cohomological contractions,
evaluating their chamber constant terms, and identifying the resulting
quasimodular form.  We retain labeled half-edges, so every multiplicity is
combined with the factors $1/r_i!$ in \eqref{eq:tensor-current}.

\subsection{Okounkov's one-variable generators}

For multiple-$q$-zeta series we use Okounkov's notation
~\cite[Section~1.2, Equation~(2)]{Okounkov2014}, with a bold letter
$\mathbf Z$.  Thus $\mathbf Z(2)$, $\mathbf Z(4)$, and $\mathbf Z(6)$
are $q$-series, distinct from the quasi-elliptic function $Z(z)$ used above.
The three one-variable series needed below are
\begin{align}
\OZ{2}
&=\sum_{n>0}\frac{q^n}{(1-q^n)^2},
\label{eq:Ok-Z2}\\
\OZ{4}
&=\sum_{n>0}\frac{q^{2n}}{(1-q^n)^4},
\label{eq:Ok-Z4}\\
\OZ{6}
&=\sum_{n>0}\frac{q^{3n}}{(1-q^n)^6}.
\label{eq:Ok-Z6}
\end{align}
We will also use the depth-two series
\begin{equation}\label{eq:Ok-Z22}
\OZ{2,2}
=\sum_{n_1>n_2>0}
\frac{q^{n_1+n_2}}{(1-q^{n_1})^2(1-q^{n_2})^2}.
\end{equation}
Separating the two summation indices according as
$n_1>n_2$, $n_1<n_2$, or $n_1=n_2$ gives
\begin{equation}\label{eq:Ok-product22}
\OZ{2}^2=2\OZ{2,2}+\OZ{4}.
\end{equation}  We have
\begin{equation}\label{eq:T2-OkZ2}
T_2=\OZ{2}=\frac{1-E_2}{24}.
\end{equation}
The next elementary identity will be used repeatedly.

\begin{lemma}\label[lemma]{lem:DZ2-correct}
\begin{equation}\label{eq:DZ2-correct}
\Dq\OZ{2}
=5\OZ{4}-2\OZ{2}^{2}+\OZ{2}.
\end{equation}
Equivalently,
\begin{equation}\label{eq:DZ2-Eisenstein}
\Dq\OZ{2}=\frac{E_4-E_2^2}{288}.
\end{equation}
\end{lemma}

\begin{proof}
Write $\sigma_s(m)=\sum_{d\mid m}d^s$.  The definitions give
\[
\OZ{2}=\sum_{m\geq1}\sigma_1(m)q^m,
\qquad
\OZ{4}=\frac16\sum_{m\geq1}\bigl(\sigma_3(m)-\sigma_1(m)\bigr)q^m.
\]
Hence
\begin{equation}\label{eq:Ok-Eisenstein-relations}
E_2=1-24\OZ{2},
\qquad
E_4=1+240\OZ{2}+1440\OZ{4}.
\end{equation}
Ramanujan's differential identity
\begin{equation}\label{eq:Ramanujan-E2}
\Dq E_2=\frac{E_2^2-E_4}{12}
\end{equation}
(see, for example,~\cite{KanekoZagier1995}), together with \eqref{eq:Ok-Eisenstein-relations}, gives \eqref{eq:DZ2-correct} and \eqref{eq:DZ2-Eisenstein}.
\end{proof}

The following four constant terms are the scalar evaluations needed below.

\begin{lemma}\label[lemma]{lem:elementary-CTs}
In the chamber $|q|<|x_i/y|<1$ and $|q|<|x_2/x_1|<1$,
\begin{align}
[x^0]\zhat(x/y)^2
&=-2T_2,
\label{eq:CT-Zhat2}\\*
[x^0]\zhat(x/y)D_x^2\zhat(x/y)
&=-2\Dq T_2,
\label{eq:CT-Zhat-D2Zhat}\\*
[x_1^0x_2^0]\zhat(x_1/y)\zhat(x_2/y)P(x_2/x_1)
&=-\Dq T_2.
\label{eq:CT-ZZP}
\end{align}
Also define the series
\begin{equation}\label{eq:H-definition}
\mathcal H(q)
=\frac12[x_1^0x_2^0]
\zhat(x_1/y)\zhat(x_2/y)P(x_2/x_1)^2.
\end{equation}

It satisfies
\begin{align}
\mathcal H(q)
={}&-\frac54\OZ{2}^2-\frac54\OZ{4}
+\frac{10}{3}\OZ{2}^3
-5\OZ{2}\OZ{4}
-\frac{35}{6}\OZ{6}.
\label{eq:H-polynomial}
\end{align}
\end{lemma}

\begin{proof}
The change of variables $x_i\mapsto x_i/y$ permits us to set $y=1$.
For $n>0$, the Fourier coefficients of $\zhat$ are
\[
\zhat_n=\frac1{1-q^n},\qquad
\zhat_{-n}=-\frac{q^n}{1-q^n}.
\]
Thus the modes $n$ and $-n$ give
\[
[x^0]\zhat(x)^2
=-2\sum_{n>0}\frac{q^n}{(1-q^n)^2}=-2T_2.
\]
Applying $D_x^2$ to the second factor inserts $n^2$ in each summand.
Differentiating the definition of $T_2$ in \eqref{eq:loop-series-def} gives
\[
\Dq T_2=\sum_{n>0}\frac{n^2q^n}{(1-q^n)^2},
\]
which proves \eqref{eq:CT-Zhat-D2Zhat}.

For \eqref{eq:CT-ZZP}, a simultaneous constant term forces the
three mode indices to be $(r,-r,r)$, in the order
$\zhat(x_1),\zhat(x_2),P(x_2/x_1)$.  Adding the orientations
$r=n$ and $r=-n$ gives
\begin{align*}
[x_1^0x_2^0]\zhat(x_1)\zhat(x_2)P(x_2/x_1)
&=-\sum_{n>0}\frac{nq^n(1+q^n)}{(1-q^n)^3}\\
&=-\sum_{n,k>0}nk^2q^{nk}
=-\sum_{n,k>0}n^2kq^{nk}
=-\Dq T_2.
\end{align*}
Here $t(1+t)/(1-t)^3=\sum_{k>0}k^2t^k$, and the middle equality
interchanges $n$ and $k$.  These rearrangements are valid coefficientwise:
only finitely many pairs $(n,k)$ contribute to any fixed power of $q$.

For the fourth constant term, we compute the coefficients of $\mathcal H$
and then determine its quasimodular expression.  Set $y=1$ and let $b_r=[u^r]P(u)^2$.
The simultaneous constant term imposes mode indices $(r,-r,r)$, so
\begin{equation}\label{eq:H-mode-extraction}
\mathcal H(q)
=\frac12\sum_{r\neq0}
\zhat_r(q)\zhat_{-r}(q)\,[u^r]P(u)^2,
\end{equation}
where $\zhat_r=(1-q^r)^{-1}$ and
$\zhat_{-r}=-q^r/(1-q^r)$ for $r>0$.
The coefficients $P_r(q)=[u^r]P(u)$ satisfy $P_{-r}=q^rP_r$, which gives
$b_{-r}=q^rb_r$.  Splitting the convolution for $b_r$ into two positive
modes and two modes of opposite signs yields, for $r>0$,
\begin{align}
b_r={}&\sum_{a=1}^{r-1}
\frac{a(r-a)}{(1-q^a)(1-q^{r-a})}
+2\sum_{a\geq1}
\frac{a(a+r)q^a}{(1-q^a)(1-q^{a+r})},
\label{eq:H-convolution}\\
\mathcal H(q)={}&-\frac12\sum_{r\geq1}
\frac{q^r(1+q^r)}{(1-q^r)^2}\,b_r.
\label{eq:H-positive-modes}
\end{align}
These formulas give a finite calculation modulo $q^{Q+1}$: take
$r\leq Q$, and in the second sum of \eqref{eq:H-convolution} take
$a\leq Q-r$.  Since $b_1=4q+O(q^2)$ and $b_2=1+O(q)$, the first
nonzero coefficient is $-2-1/2=-5/2$ at $q^2$.  Continuing the same
finite expansion gives
\begin{equation}\label{eq:H-expansion}
\mathcal H(q)
=-\frac52q^2-20q^3-75q^4-200q^5-450q^6-840q^7+O(q^8).
\end{equation}

Every monomial in the defining product has total exponent zero in
$y,x_1,x_2$.  Its constant term in $x_1,x_2$ is therefore also its
constant term in all three variables, as required in \Cref{thm:GM}.
The two factors $\zhat$ have mixed weight at most one each, and the
two factors $P$ have weight two each.  Hence their product has mixed
weight at most $1+1+2+2=6$.  Applying \Cref{thm:GM} directly gives
$\mathcal H\in\Fil_{\leq6}\QM$, independently of the
two-point calculation below.
To specify a basis for this filtered space, besides
\eqref{eq:Ok-Eisenstein-relations} use the coefficient identity
\[
\OZ{6}=\frac1{120}\sum_{n\geq1}
\bigl(\sigma_5(n)-5\sigma_3(n)+4\sigma_1(n)\bigr)q^n,
\]
which follows by expanding $t^3/(1-t)^6$.
It implies
\begin{equation}\label{eq:Ok-E6-relation}
E_6=1-504\OZ{2}-15120\OZ{4}-60480\OZ{6}.
\end{equation}
The triangular change of generators from $E_2,E_4,E_6$ therefore shows
that
\[
1,\quad\OZ{2},\quad\OZ{2}^2,\quad\OZ{4},\quad
\OZ{2}^3,\quad\OZ{2}\OZ{4},\quad\OZ{6}
\]
is a basis of $\Fil_{\leq6}\QM$.  Since $\mathcal H$ has zero constant
and linear coefficients, and only $1$ and $\OZ{2}=q+O(q^2)$ detect
those coefficients in this basis, write
\[
\mathcal H
=c_1\OZ{2}^2+c_2\OZ{4}+c_3\OZ{2}^3
+c_4\OZ{2}\OZ{4}+c_5\OZ{6}.
\]
The coefficients of $q^2,\ldots,q^6$ give the exact system
\begin{equation}\label{eq:H-reconstruction-matrix}
\begin{pmatrix}
1&1&0&0&0\\
6&4&1&1&1\\
17&11&9&7&6\\
38&20&39&27&21\\
70&40&120&76&57
\end{pmatrix}
\begin{pmatrix}c_1\\c_2\\c_3\\c_4\\c_5\end{pmatrix}
=
\begin{pmatrix}-5/2\\-20\\-75\\-200\\-450\end{pmatrix}.
\end{equation}
Its determinant is $72$, and its unique solution is
\[
(c_1,c_2,c_3,c_4,c_5)
=\left(-\frac54,-\frac54,\frac{10}{3},-5,-\frac{35}{6}\right).
\]
Because membership in this five-dimensional space has already been
proved, invertibility establishes \eqref{eq:H-polynomial} as an
identity of complete series.  The coefficient at $q^7$ is an additional
arithmetic check, not an assumption in the reconstruction.
\end{proof}

\subsection{The rank insertion and the loop coefficient}

The rank insertion gives the simplest check of the coefficients in
\Cref{sec:algorithm}: it has one current and only two labeled Wick
graphs.  The answer is the known rank identity of
~\cite[Remark~4.12; see also Proposition~4.11]{QinYu2018}.

\begin{proposition}\label[proposition]{prop:ch0-check}
Assume that $K_X$ is numerically trivial.  For every line bundle $L$ on
$X$, the algorithm for $k=0$ gives
\begin{equation}\label{eq:ch0-formula}
\left\langle\ch_0^L\right\rangle'
=\chi(X)T_2
=\frac{\chi(X)}{24}(1-E_2).
\end{equation}
\end{proposition}

\begin{proof}
\Cref{alg:T-GRR,alg:F-current-data} select the single GRR pair
$(0,\1)$ and the current datum $(2,\1,(0,0),-1)$, with coefficient
$-1/2!$.  In \Cref{alg:F-graphs}, the two-singleton graph contributes
$\int_X\1=0$, while the loop contributes $\int_X\eX=\chi(X)$.
Its scalar factor in \Cref{alg:F-kernels} is $-2T_2$, independent of
the current variable, so \Cref{alg:F-sum} gives
\[
\left\langle\ch_0^L\right\rangle'
=\left(-\frac1{2!}\right)
\bigl(0+\chi(X)(-2T_2)\bigr)=\chi(X)T_2.
\]
Equation \eqref{eq:T2-OkZ2} gives the Eisenstein expression.
Thus the known rank series is recovered with the coefficient $1/2!$
from the current and the factor $-2T_2$ from its labeled loop, with
no additional division by two.
\end{proof}

\subsection{The one-point series $\langle\ch_2^L\rangle'$}

\begin{proposition}\label[proposition]{prop:ch2-check}
Assume that $K_X$ is numerically trivial.  For every line bundle $L$ on $X$,
\begin{align}
\left\langle\ch_2^L\right\rangle'
={}&\frac{L^2}{2}T_2
+\chi(X)\left(
-T_2^2-\frac1{12}\Dq T_2+\frac1{12}T_2
\right)
\label{eq:ch2-raw}\\
={}&\frac12\OZ{2}\,L^2
-\left(
\frac56\OZ{2}^2+\frac5{12}\OZ{4}
\right)\chi(X)
\label{eq:ch2-OkZ}\\
={}&\frac{1-E_2}{48}L^2
+\frac{6-5E_2^2-E_4}{3456}\chi(X).
\label{eq:ch2-Eisenstein}
\end{align}
\end{proposition}

\begin{proof}
Write $\ell=c_1(L)$.  The GRR branches in \Cref{alg:T-GRR} are
\[
(2,\1),\qquad (1,\ell),\qquad
\left(0,\frac{\ell^2}{2}\right).
\]
For a current with $r$ fields and cohomology class $\gamma$, a graph
with $g$ loops has cohomological contribution $\int_X\eX^g\gamma$.
It can therefore survive only when $\deg\gamma+4g=4$.  This condition identifies the possible numbers of loops and singletons
before we expand the scalar kernels in powers of $x$ and extract their
constant terms.

For $(2,\1)$, \Cref{alg:F-current-data} gives $A$ with $r=4$,
$B_1,B_2$ with $r=2$ and derivative orders $(2,0)$ and $(0,2)$,
and $C$ with $r=2$ and derivative orders $(0,0)$.  For $A$, the degree condition requires one loop and two
singletons.  There are $\binom42=6$ choices of the labeled loop, each
with cohomological contribution $\chi(X)$ and scalar constant term
\[
[x^0]\zhat(x/y)^2(-2T_2)=4T_2^2.
\]
The coefficient $c/r!$ is $-1/4!$, so their total is
$-\chi(X)T_2^2$.  The four-singleton graph contains $\int_X\1$,
and the two-loop graphs contain $\int_X\eX^2$; both vanish.

The $B_h$ and $C$ choices have cohomology class $\gamma=\eX$.  Their only surviving
partition consists of two singletons, because an additional loop would
insert another Euler class.  Including the current coefficients and
labeled multiplicities, \Cref{lem:elementary-CTs} gives
\[
\begin{array}{c|c|c}
\text{data}&\text{coefficient including multiplicity}
 &\text{total contribution}\\ \hline
B_1,B_2&2/(24\cdot2!)&-\chi(X)\Dq T_2/12\\
C&-1/(12\cdot2!)&\chi(X)T_2/12.
\end{array}
\]
Consequently the sum in \Cref{alg:F-sum} is
\begin{equation}\label{eq:G2-one-check}
\Omega_{q,y}(\GG_2(\1))
=\chi(X)\left(
-T_2^2-\frac1{12}\Dq T_2+\frac1{12}T_2
\right).
\end{equation}

For $(1,\ell)$, the Euler-class correction data vanish because
$\eX\ell=0$.  The remaining choice has $r=3$ and either no loop or
one loop, giving $\int_X\ell$ or $\int_X\eX\ell$, respectively.
Both are zero, and hence
\begin{equation}\label{eq:G1ell-zero}
\Omega_{q,y}(\GG_1(\ell))=0.
\end{equation}
For $(0,\ell^2/2)$, the degree condition instead requires no loop.
The two-singleton graph has contribution $L^2/2$ and constant term
$-2T_2$, so
\begin{equation}\label{eq:G0ell2-check}
\Omega_{q,y}\!\left(\GG_0\!\left(\frac{\ell^2}{2}\right)\right)
=\left(-\frac1{2!}\right)\frac{L^2}{2}(-2T_2)
=\frac{L^2}{2}T_2.
\end{equation}
Adding the three branches as in \Cref{alg:T-sum} proves
\eqref{eq:ch2-raw}.  Substitution of \Cref{lem:DZ2-correct} and
\eqref{eq:T2-OkZ2} gives \eqref{eq:ch2-OkZ} and
\eqref{eq:ch2-Eisenstein}.
\end{proof}

\subsection{The two-point series $\langle\ch_1^{L_1}\ch_1^{L_2}\rangle'$}

\begin{proposition}\label[proposition]{prop:ch1ch1-check}
Assume that $K_X$ is numerically trivial.  For line bundles $L_1,L_2$ on $X$,
\begin{equation}\label{eq:ch1ch1-raw}
\left\langle
\ch_1^{L_1}\ch_1^{L_2}
\right\rangle'
=(L_1\cdot L_2)\Dq\OZ{2}+\chi(X)\mathcal H(q),
\end{equation}
where $\mathcal H$ is evaluated in \Cref{lem:elementary-CTs}.
Consequently,
\begin{align}
\left\langle
\ch_1^{L_1}\ch_1^{L_2}
\right\rangle'
={}&(L_1\cdot L_2)
\bigl(5\OZ{4}-2\OZ{2}^2+\OZ{2}\bigr)
\notag\\*
&+\chi(X)\bigl(
-\frac54\OZ{2}^2-\frac54\OZ{4}
+\frac{10}{3}\OZ{2}^3
-5\OZ{2}\OZ{4}
-\frac{35}{6}\OZ{6}
\bigr).
\label{eq:ch1ch1-polynomial}
\end{align}
\end{proposition}

\begin{proof}
Put $\ell_i=c_1(L_i)$.  The GRR reduction gives
\[
\ch_1(L_i^{[n]})=G_1(\1,n)+G_0(\ell_i,n).
\]
For $j=1$, the two correction terms in
\eqref{eq:algorithm-current-summary} contain one Heisenberg field.
Since $\tau_1$ is the identity, \eqref{eq:heisenberg-field} gives
\begin{align}
\Vcur_1(\eX;x)&=\sum_{m\ne0}\Nak_m(\eX)x^{-m},
\label{eq:length-one-current}\\
\Scur_1(\eX;x)&=D_x^2\Vcur_1(\eX;x)
=\sum_{m\ne0}m^2\Nak_m(\eX)x^{-m}.
\label{eq:length-one-derivative}
\end{align}
Consequently their combined constant term is
\begin{equation}\label{eq:length-one-correction-zero}
[x^0]\left(\frac1{24}\Scur_1(\eX;x)
-\frac1{12}\Vcur_1(\eX;x)\right)
=\sum_{m\ne0}\left(\frac{m^2}{24}-\frac1{12}\right)
\Nak_m(\eX)[x^0]x^{-m}=0.
\end{equation}
This also removes either correction from a product of currents: the
other current depends on a different variable, so, for any such
operator-valued series $A(x_2)$,
\[
[x_1^0x_2^0]\bigl(B(x_1)A(x_2)\bigr)
=([x_1^0]B(x_1))([x_2^0]A(x_2))=0
\]
when $B$ is either correction in \eqref{eq:length-one-correction-zero}.
The order of the operators is unchanged.  Thus
\Cref{lem:trace-CT-interchange} reduces the calculation to the choices
\[
(3,\1,(0,0,0),-1)
\quad\text{and}\quad
(2,\ell_i,(0,0),-1).
\]
For either mixed GRR branch choice, a connected component containing
$\ell_i$ contributes an integral of degree $2+4g$, which cannot equal
four.  This also rules out disconnected graphs, because the component
containing $\ell_i$ already vanishes.  Thus these branches are
zero.

For the current selection with $(r_1,r_2)=(2,2)$ and
$(\gamma_1,\gamma_2)=(\ell_1,\ell_2)$, any nonzero graph must be connected,
and its contribution $\int_X\eX^g\ell_1\ell_2$ requires $g=0$.
With two vertices this means exactly one intervertex edge and no loops;
the two remaining half-edges are singletons.  There are $2\cdot2=4$
labeled choices.  Multiplying by the coefficient
$\prod_i c_i/r_i!=1/(2!2!)$ gives $4/(2!2!)=1$.
Their cohomological contribution is $L_1\cdot L_2$, and their kernel
is $-\zhat(x_1/y)\zhat(x_2/y)P(x_2/x_1)$.  Therefore
\Cref{lem:elementary-CTs} gives
\begin{equation}\label{eq:V2V2-check}
(L_1\cdot L_2)\Dq T_2.
\end{equation}

For the current selection with $(r_1,r_2)=(3,3)$ and
$(\gamma_1,\gamma_2)=(\1,\1)$, let $m$ be the number of intervertex
edges and $l_i$ the number of loops at vertex $i$.  A connected graph
has $g=m+l_1+l_2-1$, so its cohomological contribution can be nonzero
only if $m+l_1+l_2=2$.  This leaves two possibilities.  If $m=2$,
then $l_1=l_2=0$, and one singleton remains at each vertex.  If $m=1$,
there is one loop, exhausting the half-edges at one vertex.  If the
loop is at vertex $1$, its scalar factor satisfies
\[
[x_1^0]\bigl((-2T_2)(-P(x_2/x_1))\zhat(x_2/y)^2\bigr)
=2T_2\zhat(x_2/y)^2[x_1^0]P(x_2/x_1)=0,
\]
because $P$ has no term of exponent zero.  The same argument applies
with the two vertices exchanged.  A disconnected graph can survive the degree
test only with one loop and one singleton at each vertex; its constant
term vanishes because it contains
$[x_1^0]\zhat(x_1/y)\,[x_2^0]\zhat(x_2/y)=0$.

Hence only the two-intervertex-edge graphs contribute.  Choosing the
singleton at each vertex and pairing the remaining half-edges gives
$3\cdot3\cdot2!=18$ labeled graphs.  Their combined coefficient is
$18\prod_i c_i/r_i!=18/(3!3!)=1/2$, their cohomological contribution is $\chi(X)$,
and the two negative edge signs cancel.  Thus
\begin{equation}\label{eq:V3V3-check}
\frac{\chi(X)}2[x_1^0x_2^0]
\zhat(x_1/y)\zhat(x_2/y)P(x_2/x_1)^2
=\chi(X)\mathcal H(q).
\end{equation}
Summing the four GRR branches proves \eqref{eq:ch1ch1-raw}. Finally, substitute \eqref{eq:H-polynomial} from
\Cref{lem:elementary-CTs} and \eqref{eq:DZ2-correct} from
\Cref{lem:DZ2-correct} into \eqref{eq:ch1ch1-raw} to obtain
\eqref{eq:ch1ch1-polynomial}.
\end{proof}

The coefficients of $q$ have a direct geometric check.  Since
$X^{[1]}=X$ and $L^{[1]}=L$, the defining integrals are
$\int_X\ch_2(L)=L^2/2$ and
$\int_X\ch_1(L_1)\ch_1(L_2)=L_1\cdot L_2$.
Normalization does not change these coefficients, because both
insertion series have zero constant term.  This agrees with
\eqref{eq:ch2-raw} and \eqref{eq:ch1ch1-raw}, using
$T_2=q+O(q^2)$, $\Dq T_2=q+O(q^2)$, and
$\mathcal H=O(q^2)$.

\begin{remark}[Comparison with two formulas in the literature]\label[remark]{rem:printed-formulas}
The low-degree calculations in \Cref{prop:ch2-check,prop:ch1ch1-check} permit two precise comparisons with
~\cite{AlhwaimelQin2024,Alhwaimel2025}.  Both discrepancies are confined to
the identities compared in \eqref{eq:literature-DZ2} and \eqref{eq:literature-P}.
The comparisons below concern the preprint versions of these papers
and use their equation numbering.

First, ~\cite[Equation~(5.12)]{AlhwaimelQin2024} uses
\begin{equation}\label{eq:literature-DZ2}
\Dq\OZ{2}
=3\OZ{4}-\OZ{2,2}+\OZ{2}
=\frac72\OZ{4}-\frac12\OZ{2}^2+\OZ{2}.
\end{equation}
With the definitions of the Okounkov series used there, however, the relations \eqref{eq:Ok-Eisenstein-relations} and Ramanujan's identity \eqref{eq:Ramanujan-E2} give
\begin{equation}\label{eq:literature-DZ2-correction}
\Dq\OZ{2}
=5\OZ{4}-2\OZ{2}^2+\OZ{2},
\end{equation}
which is \eqref{eq:DZ2-correct}.  A direct \(q^3\)-check gives
\(3\sigma_1(3)=12\) on the left and \(15\) on the right.  Thus the first
equality in \eqref{eq:literature-DZ2} is the source of the discrepancy; the
conversion using \eqref{eq:Ok-product22} is not.

For numerically trivial \(K_X\), the formula in
~\cite[Theorem~6.2]{AlhwaimelQin2024} specializes to
\begin{equation}\label{eq:literature-ch2-specialization}
\frac12\OZ{2}L^2
-\left(\frac{23}{24}\OZ{2}^2+\frac7{24}\OZ{4}\right)\chi(X).
\end{equation}
Our calculation in \Cref{prop:ch2-check} gives \eqref{eq:ch2-OkZ}.  The difference between
\eqref{eq:ch2-OkZ} and \eqref{eq:literature-ch2-specialization} is
\[
\frac18\bigl(\OZ{2}^2-\OZ{4}\bigr)\chi(X).
\]
This is exactly the correction obtained by replacing
\eqref{eq:literature-DZ2} with
\eqref{eq:literature-DZ2-correction} in the derivative term.  This explains
both why the formulas differ and why the formula in
\Cref{prop:ch2-check} has the stated coefficients.

Second, ~\cite[Definition~5.6, Equation~(5.22)]{Alhwaimel2025}
defines $h_{1,1}^{(2)}=-S_1-\frac12S_2$, where
\begin{align*}
S_1(q)&=\sum_{i,j>0}
\frac{j(i+j)(q^{i+j}+q^{2i+j})}
{(1-q^i)^2(1-q^j)(1-q^{i+j})},\\
S_2(q)&=\sum_{i,j>0}
\frac{ij(q^{i+j}+q^{2i+2j})}
{(1-q^i)(1-q^j)(1-q^{i+j})^2}.
\end{align*}
To identify these series with our constant term, substitute
\eqref{eq:H-convolution} into \eqref{eq:H-positive-modes}.
In the finite convolution set $r=i+j$ and $a=i$; its contribution
is $-S_2/2$.  In the second convolution set $r=i$ and $a=j$;
its contribution is $-S_1$.  Thus
\begin{equation}\label{eq:H-literature-identification}
\mathcal H(q)=-S_1(q)-\frac12S_2(q)=h_{1,1}^{(2)}(q).
\end{equation}
These reindexings are valid in $\Q[[q]]$, since every resulting
summand has $q$-order at least $i+j$ and hence each coefficient
receives contributions from only finitely many pairs.
The terms with $i=j=1$ consequently give
\begin{equation}\label{eq:literature-h-coefficient}
[q^2]h_{1,1}^{(2)}=-2-\frac12=-\frac52.
\end{equation}
On the other hand, for $P$ in ~\cite[Equation~(5.32)]{Alhwaimel2025}:
\begin{equation}\label{eq:literature-P}
P(q)=
\frac54\OZ{2}^2+\frac54\OZ{4}
-\frac{10}{3}\OZ{2}^3+5\OZ{2}\OZ{4}
+\frac{35}{6}\OZ{6},
\end{equation}
whose $q^2$-coefficient is $5/2$, in contrast with
\eqref{eq:literature-h-coefficient}.  More precisely, the exact
identification \eqref{eq:H-literature-identification} and the
reconstruction \eqref{eq:H-polynomial} show that
\[
h_{1,1}^{(2)}(q)=\mathcal H(q)=-P(q).
\]
Thus the polynomial conversion in \eqref{eq:literature-P} has the
opposite sign from the defining double series.  The coefficient at
$q^2$ detects the discrepancy, while the convolution and invertible
system \eqref{eq:H-reconstruction-matrix} establish the complete
series identity.
\end{remark}

\section*{Acknowledgments}

This project began when the authors asked GPT-5.6 Sol to prove Qin's conjecture.  The model proposed a promising proof strategy.  Most of the formal arguments in this paper were initially generated by GPT-5.6 Sol.  All AI-assisted parts were independently verified and revised by the authors.  The authors take full responsibility for all mathematical claims and any remaining errors.  The authors thank Angela Gibney, Daniel Krashen and Zhenbo Qin for helpful comments on an early draft of the paper.

\bibliographystyle{alpha}
\bibliography{references}

@article {Alhwaimel2025,
  AUTHOR = {Alhwaimel, Mazen M.},
  TITLE = {Toward {Q}in's conjecture on {H}ilbert schemes of points and quasi-modular forms},
  JOURNAL = {Manuscripta Math.},
  VOLUME = {176},
  YEAR = {2025},
  PAGES = {31},
  DOI = {10.1007/s00229-025-01630-1},
  EPRINT = {2407.01554},
  ARCHIVEPREFIX = {arXiv},
  PRIMARYCLASS = {math.AG}
}

@article {AlhwaimelQin2024,
  AUTHOR = {Alhwaimel, Mazen M. and Qin, Zhenbo},
  TITLE = {Hilbert schemes of points on surfaces and multiple {$q$}-zeta values},
  JOURNAL = {Pure Appl. Math. Q.},
  VOLUME = {20},
  YEAR = {2024},
  NUMBER = {6},
  PAGES = {2615--2646},
  EPRINT = {2310.10812},
  ARCHIVEPREFIX = {arXiv},
  PRIMARYCLASS = {math.AG}
}

@article {Carlsson2012,
  AUTHOR = {Carlsson, Erik},
  TITLE = {Vertex operators and quasimodularity of {C}hern numbers on the {H}ilbert scheme},
  JOURNAL = {Adv. Math.},
  VOLUME = {229},
  YEAR = {2012},
  NUMBER = {5},
  PAGES = {2888--2907},
  DOI = {10.1016/j.aim.2012.01.005}
}

@article {CarlssonOkounkov2012,
  AUTHOR = {Carlsson, Erik and Okounkov, Andrei},
  TITLE = {Exts and vertex operators},
  JOURNAL = {Duke Math. J.},
  VOLUME = {161},
  YEAR = {2012},
  NUMBER = {9},
  PAGES = {1797--1815},
  DOI = {10.1215/00127094-1593380}
}

@article {Gottsche1990,
  AUTHOR = {G{\"o}ttsche, Lothar},
  TITLE = {The {B}etti numbers of the {H}ilbert scheme of points on a smooth projective surface},
  JOURNAL = {Math. Ann.},
  VOLUME = {286},
  YEAR = {1990},
  NUMBER = {1--3},
  PAGES = {193--207},
  DOI = {10.1007/BF01453572}
}

@article {GoujardMoller2020,
  AUTHOR = {Goujard, Elise and M{\"o}ller, Martin},
  TITLE = {Counting {F}eynman-like graphs: quasimodularity and {S}iegel--{V}eech weight},
  JOURNAL = {J. Eur. Math. Soc. (JEMS)},
  VOLUME = {22},
  YEAR = {2020},
  NUMBER = {2},
  PAGES = {365--412},
  DOI = {10.4171/JEMS/924}
}

@article {Grojnowski1996,
  AUTHOR = {Grojnowski, Ian},
  TITLE = {Instantons and affine algebras. {I}. The {H}ilbert scheme and vertex operators},
  JOURNAL = {Math. Res. Lett.},
  VOLUME = {3},
  YEAR = {1996},
  NUMBER = {2},
  PAGES = {275--291},
  DOI = {10.4310/MRL.1996.v3.n2.a12}
}

@incollection {KanekoZagier1995,
  AUTHOR = {Kaneko, Masanobu and Zagier, Don},
  TITLE = {A generalized {J}acobi theta function and quasimodular forms},
  BOOKTITLE = {The moduli space of curves},
  SERIES = {Progr. Math.},
  VOLUME = {129},
  PUBLISHER = {Birkh{\"a}user Boston},
  ADDRESS = {Boston, MA},
  YEAR = {1995},
  PAGES = {165--172}
}

@article {LiQinWang2002a,
  AUTHOR = {Li, Wei-Ping and Qin, Zhenbo and Wang, Weiqiang},
  TITLE = {Vertex algebras and the cohomology ring structure of {H}ilbert schemes of points on surfaces},
  JOURNAL = {Math. Ann.},
  VOLUME = {324},
  YEAR = {2002},
  NUMBER = {1},
  PAGES = {105--133},
  DOI = {10.1007/s002080200330}
}

@article {LiQinWang2002b,
  AUTHOR = {Li, Wei-Ping and Qin, Zhenbo and Wang, Weiqiang},
  TITLE = {Hilbert schemes and {$W$} algebras},
  JOURNAL = {Int. Math. Res. Not. IMRN},
  YEAR = {2002},
  NUMBER = {27},
  PAGES = {1427--1456},
  DOI = {10.1155/S1073792802110129}
}

@article {Nakajima1997,
  AUTHOR = {Nakajima, Hiraku},
  TITLE = {Heisenberg algebra and {H}ilbert schemes of points on projective surfaces},
  JOURNAL = {Ann. of Math. (2)},
  VOLUME = {145},
  YEAR = {1997},
  NUMBER = {2},
  PAGES = {379--388},
  DOI = {10.2307/2951818}
}

@article {Okounkov2014,
  AUTHOR = {Okounkov, Andrei},
  TITLE = {Hilbert schemes and multiple {$q$}-zeta values},
  JOURNAL = {Funct. Anal. Appl.},
  VOLUME = {48},
  YEAR = {2014},
  NUMBER = {2},
  PAGES = {138--144},
  DOI = {10.1007/s10688-014-0054-z}
}

@book {Qin2018,
  AUTHOR = {Qin, Zhenbo},
  TITLE = {Hilbert schemes of points and infinite dimensional {L}ie algebras},
  SERIES = {Mathematical Surveys and Monographs},
  VOLUME = {228},
  PUBLISHER = {American Mathematical Society},
  ADDRESS = {Providence, RI},
  YEAR = {2018},
  DOI = {10.1090/surv/228}
}

@article {QinYu2018,
  AUTHOR = {Qin, Zhenbo and Yu, Fei},
  TITLE = {On {O}kounkov's conjecture connecting {H}ilbert schemes of points and multiple {$q$}-zeta values},
  JOURNAL = {Int. Math. Res. Not. IMRN},
  YEAR = {2018},
  NUMBER = {2},
  PAGES = {321--361},
  DOI = {10.1093/imrn/rnw244}
}

@article {ShenQin2020,
  AUTHOR = {Shen, Zhongyan and Qin, Zhenbo},
  TITLE = {Hilbert schemes of points and quasi-modularity},
  JOURNAL = {Pure Appl. Math. Q.},
  VOLUME = {16},
  YEAR = {2020},
  NUMBER = {5},
  PAGES = {1673--1706},
  DOI = {10.4310/PAMQ.2020.v16.n5.a11}
}

\end{document}